\documentclass[12pt]{article}

\usepackage[margin=1in]{geometry}
\usepackage[utf8]{inputenc}

\usepackage{amsmath,amsthm,amssymb,enumerate,multirow, booktabs, color, scrextend,url}
\usepackage[dvipsnames]{xcolor}

\newcommand{\Mhm}[1]{\textcolor{purple}{#1}}
\newcommand{\rev}[1]{\textcolor{black}{#1}}

\usepackage[normalem]{ulem}
\usepackage{comment}

\usepackage[caption=false]{subfig}
\usepackage{graphicx}

\usepackage{footnote}
\makesavenoteenv{algorithm}
\makesavenoteenv{align}
\makesavenoteenv{tabular}
\makesavenoteenv{table}

\usepackage{algorithm}
\usepackage{algorithmic}
\usepackage{setspace}
\usepackage[margin=1in]{geometry}
\usepackage[fleqn,tbtags]{mathtools}
\usepackage[english]{babel}

\usepackage{cleveref}
\usepackage[numbers]{natbib} 

\title{A norm minimization-based convex vector optimization algorithm\thanks{This work was funded by T\"UB{\.I}TAK (Scientific \& Technological
		Research Council of Turkey), Project No. 118M479.}}
\author{\c{C}a\u{g}{\i}n Ararat\thanks{Bilkent University, Department of Industrial Engineering, Ankara, Turkey, cararat@bilkent.edu.tr.}
	\and
	Firdevs Ulus\thanks{Bilkent University, Department of Industrial Engineering, Ankara, Turkey, firdevs@bilkent.edu.tr.}
	\and
	Muhammad Umer\thanks{Bilkent University, Department of Industrial Engineering, Ankara, Turkey, muhammad.umer@bilkent.edu.tr.}
}
\date{\today}

\allowdisplaybreaks

\addto\captionsenglish {}
\makeatletter \renewenvironment{proof}[1][\proofname] {\par\pushQED{\qed}\normalfont\topsep6\p@\@plus6\p@\relax\trivlist\item[\hskip\labelsep\bfseries#1\@addpunct{.}]\ignorespaces}{\popQED\endtrivlist\@endpefalse} \makeatother

\newcommand{\N}{\mathbb{N}}
\newcommand{\R}{\mathbb{R}}

\newcommand{\out}{\text{out}}
\newcommand{\inn}{\text{in}}

\DeclareMathOperator{\wMin}{wMin}
\DeclareMathOperator{\Min}{Min}

\DeclareMathOperator{\cl}{cl}
\DeclareMathOperator{\bd}{bd}
\DeclareMathOperator{\Int}{int}
\DeclareMathOperator{\conv}{conv}
\DeclareMathOperator{\cone}{cone}
\DeclareMathOperator{\sgn}{sgn}
\DeclareMathOperator{\recc}{recc}
\DeclareMathOperator{\rank}{rank}

\newcommand{\dom}{{\rm dom\,}}

\newcommand{\norm}[1]{\lVert#1\rVert}
\newcommand{\abs}[1]{\lvert#1\rvert}
\newtheorem{theorem}{Theorem}
\numberwithin{equation}{section}
\numberwithin{theorem}{section}
\newtheorem{proposition}[theorem]{Proposition}
\newtheorem{lemma}[theorem]{Lemma}
\newtheorem{corollary}[theorem]{Corollary}
\newtheorem{definition}[theorem]{Definition}
\newtheorem{remark}[theorem]{Remark}

\newtheorem{assumption}[theorem]{Assumption}
\newtheorem{notation}[theorem]{Notation}
\newtheorem{example}[theorem]{Example}

\newcounter{algo}
\newcommand{\of}[1]{\ensuremath{\left( #1 \right)}}
\newcommand{\cb}[1]{\ensuremath{ \left\{ #1 \right\} }}

\begin{document}
\maketitle
\thispagestyle{empty}

\begin{abstract}
We propose an algorithm to generate inner and outer polyhedral approximations to the upper image of a bounded convex vector optimization problem. It is an outer approximation algorithm and is based on solving norm-minimizing scalarizations. Unlike Pascolleti-Serafini scalarization used in the literature for similar purposes, it does not involve a direction parameter. Therefore, the algorithm is free of direction-biasedness. We also propose a modification of the algorithm by introducing a suitable compact subset of the upper image, which helps in proving for the first time the finiteness of an algorithm for convex vector optimization. The computational performance of the algorithms is illustrated using some of the benchmark test problems, which shows promising results in comparison to a similar algorithm that is based on Pascoletti-Serafini scalarization.\\
\\[-5pt]
\\[-5pt]
\end{abstract}

\section{Introduction}

In multiobjective optimization, the decision-maker is supposed to consider multiple objective functions simultaneously. In general, these functions conflict in the sense that improving one objective leads to deteriorating some of the others. Consequently, there does not exist a feasible solution which can generate optimal values of all the objectives. Rather, there exists a subset of feasible solutions, called \emph{efficient solutions}, which map to the so called \emph{nondominated points} in the objective space. The image of a feasible solution is said to be nondominated if none of the objective functions can be improved in value without degrading some of the other objective values. 

In vector optimization, the objective function takes values again in a vector space, namely, the objective space. However, rather than comparing the objective function values componentwise as in the multiobjective case, a more general order relation, which is induced by an ordering cone, is used for this purpose. Clearly, multiobjective optimization can be seen as a special case where the ordering cone is the positive orthant. Assuming that the vector optimization problem (VOP) is a minimization problem with respect to an ordering cone $C$, the concept of nondominated point for a multiobjective optimization problem (MOP) is generalized to \emph{minimal point with respect to $C$} in the vector optimization case.

A special class of vector optimization problems is the linear VOPs, where the objective function is linear and the feasible region is a polyhedron. There is rich literature available discussing various methods and algorithms for solving linear VOPs. They deal with the problem by generating either the efficient solutions in the decision space \cite{evans1973revised,rudloff2017parametric} or the nondominated points in the objective space \cite{benson1998outer,lohne2011vector,lohne2017vector}. The reader is referred to the books by Ehrgott \cite{ehrgott2005mutiobjective} and by Jahn \cite{jahn2009vector} for the details of these approaches.

In 1998, Benson proposed an outer approximation algorithm for linear MOPs which generates the set of all nondominated points in the objective space rather than the set of all efficient points in the decision space \cite{benson1998outer}. 
Later, this algorithm is extended to solve linear VOPs, see \cite{lohne2011vector}. The main principle of the algorithm is that if one adds the ordering cone to the image of the feasible set, then the resulting set, called the \emph{upper image}, contains all nondominated points in its boundary. The algorithm starts with a set containing the upper image and iterates by updating the outer approximating set until it is equal to the upper image.


For nonlinear MOPs/VOPs, there is a further subdivision, namely the convex and the nonconvex problems. Note that the methods described for linear MOPs/VOPs may not be directly applicable to these classes as, in general, it is not possible to generate the set of all nondominated/minimal points in the objective space. Therefore, approximation algorithms, which approximate the set of all minimal points in the objective space, are widely explored in the literature, refer for example to the survey paper by Ruzika and Wiecek \cite{ruzika2005approximation} for the multiobjective case. 

For bounded convex vector optimization problems (CVOPs), see \Cref{sec:CVOP} for precise definitions, there are several outer approximation algorithms in the literature that work on the objective space. In \cite{ehrgott2011approximation}, the algorithm in \cite{benson1998outer} is extended for the case of convex MOPs. Another extension of Benson's algorithm for the vector optimization case is proposed in \cite{lohne2014primal}, which is a simplification and generalization of the algorithm in \cite{ehrgott2011approximation}. This algorithm has already been used for solving mixed-integer convex multiobjective optimization problems \cite{de2020solving}, as well as problems in stochastic optimization \citep{ararat2017multiobjective} and finance \citep{feinsteinrudloff2017, rudloffulus2021}. Recently, in \cite{vertexselection}, a modification of the algorithm in \cite{lohne2014primal} is proposed. The main idea of these algorithms is to generate a sequence of better approximating polyhedral supersets of the upper image until the approximation is sufficiently fine. This is done by sequentially solving some scalarization models in which the original CVOP is converted into an optimization problem with a single objective. There are many scalarization methods available in the literature for MOPs/VOPs, see for instance the book by Eichfelder \cite{eichfelder2008adaptive} as well as the recent papers \cite{burachik2017new, kasimbeyli2019comparison, meritfunction2009}.


In particular, in each iteration of the CVOP algorithms proposed in \cite{vertexselection, ehrgott2011approximation, lohne2014primal}, a \emph{Pascoletti-Serafini scalarization} \cite{pascoletti1984scalarizing}, which requires a reference point $v$ and a direction vector $d$ in the objective space as its parameters, is solved. For the algorithms in \cite{ehrgott2011approximation} and \cite{lohne2014primal}, the reference point $v$ is selected to be an arbitrary vertex of the current outer approximation of the upper image. Moreover, in \cite{ehrgott2011approximation}, the direction parameter $d$ is computed depending on the reference point $v$ together with a fixed point in the objective space, whereas it is fixed throughout the algorithm proposed in \cite{lohne2014primal}. In \cite{vertexselection}, a procedure to select a vertex $v$ as well as a direction parameter $d$, which depends on $v$ and the current approximation, is proposed.

In this study, we propose an outer approximation algorithm (\Cref{alg}) for CVOPs, which solves a norm-minimizing scalarization in each iteration. Different from Pascoletti-Serafini scalarization, it does not require a direction parameter; hence, one does not need to fix a direction parameter as in \cite{lohne2014primal}, or a point in the objective space in order to compute the direction parameter as in \cite{ehrgott2011approximation}. Moreover, when terminates, the algorithm provides the Hausdorff distance between the upper image and its outer approximation, directly.

The scalarization methods based on a norm have been frequently used in the context of MOPs, see for instance \cite{eichfelder2008adaptive}. These methods generally depend on the \emph{ideal point} at which all objectives of the MOP attain its optimal value, simultaneously. Since the ideal point is not feasible in general, the idea is to find the minimum distance from the ideal point to the image of the feasible region. One of the well-known methods is the \emph{weighted compromise programming} problem, which utilizes the $\ell_p$ norm with $p\geq 1$, see for instance \cite{lin2005min, zeleny1973compromise}. The most commonly used special case is also known as the weighted Chebyshev scalarization, where the underlying norm is taken as the $\ell_\infty$ norm, see for instance \cite{ehrgott2005multicriteria, miettinen2012nonlinear, ruzika2005approximation}. The weight vector in these scalarization problems are taken such that each component is positive. If the weight vector is taken as the vector of ones, then they are simply called compromise programming ($p\geq 1$) and Chebyshev scalarization ($p=+\infty$), respectively. 

The scalarization method that is solved in the proposed algorithm works with any norm defined on the objective space. It simply computes the distance, with respect to a fixed norm, from a given reference point in the objective space to the upper image. This is similar to compromise programming, however it has further advantages compared to it:
\begin{itemize}
	\item The reference point used in the norm-minimizing scalarization is not necessarily the ideal point, which is not well-defined for a VOP. Indeed, within the proposed algorithm, we solve it for the vertices of the outer approximation of the upper image. In weighted compromise programming, finding various nondominated points is done by varying the (nonnegative) weight parameters. It is not straightforward to generalize weighted compromise programming for a vector optimization setting, whereas this can be done directly with the proposed norm-minimizing scalarization.
\end{itemize}

We discuss some properties of the proposed scalarization under mild assumptions. In particular, we prove that if the feasible region of the VOP is solid and compact, then there exist an optimal solution to it as well as an optimal solution to its Lagrange dual. Moreover, strong duality holds between these solutions. We further prove that using a dual optimal solution, one can generate a supporting halfspace to the upper image. Note that for these results, the ordering cone is assumed to be a closed convex cone that is solid, pointed and nontrivial. However, different from the similar results regarding Pascoletti-Serafini scalarization, see for instance \cite{lohne2014primal}, the ordering cone is not necessarily polyhedral.

The main idea of \Cref{alg} is similar to the Benson-type outer approximation algorithms; iteratively, it finds better outer approximations to the upper image and stops when the approximation is sufficiently fine. As already mentioned, it solves the proposed norm-minimizing scalarization model instead of Pascoletti-Serafini scalarization. Hence, it is free of direction-biasedness. Using the properties of the norm-minimizing scalarization, we prove that the algorithm works correctly, that is, given an approximation error $\epsilon > 0$, when terminates, the algorithm returns an outer approximation to the upper image such that the Hausdorff distance between the two is less than $\epsilon$.

We also propose a modification of \Cref{alg}, namely, \Cref{alg1}. In addition to its correctness, we prove that if the feasible region is compact, then for a given approximation error $\epsilon > 0$, \Cref{alg1} stops after finitely many iterations. Note that the finiteness of outer approximation algorithms for linear VOPs are known, see for instance \cite{lohne2011vector}. Also, under compact feasible region assumption, the finiteness of an outer approximation for nonlinear (even for nonconvex) MOPs, proposed in \cite{nobakhtian2017benson}, is known. However, to the best of our knowledge, \Cref{alg1} is the first CVOP algorithm with a guarantee for finiteness. Compared to the cases of linear VOPs and nonconvex MOPs, proving the finiteness of \Cref{alg1} has the following new challenges which we address by our technical analysis:
\begin{itemize}
	\item Since the upper image is polyhedral for a linear VOP, the algorithms find exact solutions, and finiteness follows by the polyhedrality of the upper image. On the other hand, for a CVOP, we look for approximate solutions of a convex and generally nonpolyhedral upper image. Hence, the proof of finiteness requires completely different arguments.
	\item The algorithm for nonconvex MOPs in \cite{nobakhtian2017benson} constructs an outer approximation for the upper image by discarding sets of the form $\{v\}-\Int\R^q_{+}$, where $v$ is a point on the upper image (see \Cref{sec:Prelim} for precise definitions). In this case, the proof of finiteness relies on a hypervolume argument for certain small hypercubes generated by the outer approximation. In the current work, we deal with CVOPs with general ordering cones and our algorithms construct an outer approximation by intersecting certain supporting halfspaces of the upper image (instead of discarding ``point minus cone'' type sets). To prove the finiteness of \Cref{alg1}, we propose a novel hypervolume argument which exploits the relationship between these halfspaces and certain subsets of small norm balls (see \Cref{lem:Bset}). Another important challenge in using supporting halfspaces is to guarantee that the vertices of the outer approximations, which are the reference points for the scalarization models, as well as the minimal points of the upper image found by solving these scalarizations throughout the algorithm are within a compact set. Note that this is naturally the case in \cite{nobakhtian2017benson} by the structure of their outer approximations. For our proposed algorithm, we construct sufficiently large compact sets $S$ and $S_2$ such that the vertices and the corresponding minimal points of the upper image are within $S$ and $S_2$, respectively (see Lemmas \ref{lem:PcapScompact}, \ref{lem:S2} and \Cref{rem:Scontains}).
\end{itemize}

The rest of the paper is organized as follows. In \Cref{sec:Prelim}, we introduce the notation of the paper and recall some well-known concepts and results in convex analysis. In \Cref{sec:CVOP}, we present the setting for CVOP, \rev{discuss an} approximate solution concept \rev{from the literature}. This is followed by a detailed treatment of norm-minimizing scalarizations in \Cref{sec:scal}, including some duality results as well as geometric properties of optimal solutions. Sections \ref{sec:alg} and \ref{sec:alg2} are devoted to Algorithms \ref{alg} and \ref{alg1}, respectively, where we prove their correctness. The theoretical analysis of \Cref{alg1} continues in \Cref{sec:fin}, which concludes with the proof of finiteness for this algorithm. We provide several examples and discuss the computational performance of the proposed algorithms on these examples in \Cref{sec:experiments}. We conclude the paper in \Cref{sec:conclusions}.

	\section{Preliminaries} \label{sec:Prelim}

In this section, we describe the notations and definitions which will be used throughout the paper. Let $q\in\N:=\{1,2,\ldots\}$. We denote by $\R^q$ the $q$-dimensional Euclidean space. When $q=1$, we have the real line $\R:=\R^1$, and the extended real line $\overline{\R}:=\R\cup\{+\infty\}\cup\{-\infty\}$. On $\R^q$, we fix an arbitrary norm $\norm{\cdot}$, and we denote its dual norm by $\norm{\cdot}_\ast$. We will sometimes assume that $\norm{\cdot}=\norm{\cdot}_p$ is the $\ell^p$-norm on $\R^q$, where $p\in [1,+\infty]$. For $y\in\R^q$, the $\ell^p$-norm of $y$ is defined by $\norm{y}_p := (\sum_{i=1}^{q} \abs{y_i}^p)^{\frac{1}{p}}$ when $p\in [1,+\infty)$, and by $\norm{y}_{p}:=\max_{i\in\{1,\ldots,q\}}\abs{y_i}$ when $p=+\infty$. In this case, the dual norm is $\norm{\cdot}_\ast=\norm{\cdot}_{p^\prime}$, where $p^\prime\in[1,+\infty]$ is the conjugate exponent of $p$ via the relation $\frac{1}{p}+\frac{1}{p^\prime}=1$. \rev{For $\epsilon>0$, we define the closed ball $\mathbb{B}_{\epsilon}:=\{z\in\R^q\mid \norm{z}\leq \epsilon\}$ centered at the origin.}

Let $f \colon \mathbb{R}^q \rightarrow \overline{\R}$ be a convex function and $y_0 \in \R^q$ with $f(y_0) \in \R$. The set $\partial f(y_0):=\{z\in\R^q\mid \forall y\in\R^q\colon f(y)\geq  f(y_0)+z^{\mathsf{T}}(y-y_0)\}$ is called the \emph{subdifferential} of $f$ at $y_0$. 

For a set $A \subseteq \mathbb{R}^q$, we denote by $\Int A$, $\cl A$, $\bd A$, $\conv A$, $\cone A$, the interior, closure, boundary, convex hull, conic hull of $A$, respectively. A recession direction of $A$ is a vector $k \in \R^q\setminus\{0\}$ satisfying $A + \{\lambda k \in \R^q \mid \lambda > 0 \} \subseteq A$. The set of all recession directions of $A$, $\recc A = \{k\in \R^q \mid \forall a\in A, \forall \lambda \geq 0: \: a +\lambda k \in A \}$, is the \emph{recession cone} of $A$. If $A,B\subseteq\R^q$ are nonempty sets and $\lambda\in\R$, then we define the Minkowski operations $A+B:=\{y_1+y_2\mid y_1\in A, y_2\in B\}$, $\lambda A:=\{\lambda y\mid y\in A\}$, $A-B:=A+(-1)B$.

Let $C\subseteq\R^q$ be a convex cone. The set $C^+ := \{z \in \mathbb{R}^q \mid \forall y \in C : z^\mathsf{T}y \geq 0 \}$ is a closed convex cone, and it is called the \emph{dual cone} of $C$. The cone $C$ is said to be \emph{solid} if $\Int C\neq\emptyset$, \emph{pointed} if it does not contain any lines, and \emph{nontrivial} if $\{0\} \subsetneq C \subsetneq\R^q$. If $C$ is a solid pointed nontrivial cone, then the relation $\leq_C$ on $\R^q$ defined by $y_1 \leq_C y_2 \iff y_2 - y_1 \in C$ for every $y_1,y_2\in\R^q$ is a partial order. Let $X \subseteq \R^n$ be a convex set, where $n \in \N$. A function $\Gamma\colon X \rightarrow \R^q$ is said to be \emph{$C$-convex} if $\Gamma(\lambda x_1 + (1-\lambda)x_2) \leq_C \lambda \Gamma(x_1) + (1 - \lambda) \Gamma(x_2)$ for every $x_1, x_2 \in X, \lambda \in [0,1]$. In this case, the function $x\mapsto w^{\mathsf{T}}\Gamma(x)$ on $X$ is convex for every $w\in C^+$. Let $\mathcal{X}\subseteq X$. Then, the set $\Gamma(\mathcal{X}):=\{\Gamma(x)\mid x\in\mathcal{X}\}$ is the image of $\mathcal{X}$ under $\Gamma$. The function $I_{\mathcal{X}}:\R^q\to [0,+\infty]$ defined by $I_{\mathcal{X}}(x)=0$ whenever $x\in\mathcal{X}$ and by $I_{\mathcal{X}}(x)=+\infty$ whenever $x\in\R^q\setminus\mathcal{X}$ is called the \emph{indicator function} of $\mathcal{X}$.

Let $A \subseteq \R^q$ be a nonempty set. A point $y\in A$ is called a \emph{$C$-minimal element} of $A$ if $(\{y\} - C\setminus\{0\}) \cap A = \emptyset$. If the cone $C$ is solid, then $y$ is called a \emph{weakly $C$-minimal element} of $A$ if $(\{y\} -\Int C) \cap A = \emptyset$. We denote by $\Min_C(A)$ the set of all $C$-minimal elements of $A$, and by $\wMin_C(A)$ the set of all weakly $C$-minimal elements of $A$ whenever $C$ is solid. 

For each $z\in\R^q$, we define $d(z,A):=\inf_{y\in A}\norm{z-y}$. Let $B\subseteq\R^q$ be a nonempty set. We denote by $\delta^H(A,B)$ the Hausdorff distance between $A,B$. It is well-known that \cite[Proposition 3.2]{hausdorffsurvey}
\begin{align}\label{eq:Hausdorff}
\delta^H(A, B) \negthinspace=\negthinspace \max\cb{\sup_{y \in A}d(y,B), \sup_{z\in B}d(z,A)}\negthinspace=\negthinspace\inf \cb{\epsilon > 0 \mid  A \subseteq  B+\mathbb{B}_{\epsilon},\ B \subseteq A+\mathbb{B}_{\epsilon} }.
\end{align}

Suppose that $A$ is a convex set and let $y\in A$, $w\in\R^q\setminus\{0\}$. If $w^{\mathsf{T}}y=\inf_{z\in A}w^{\mathsf{T}}z$, then the set $\{z \in \R^q\mid w^{\mathsf{T}}z = w^{\mathsf{T}}y\}$ is called a \emph{supporting hyperplane} of $A$ at $y$ and the set $\{z \in \R^q\mid w^{\mathsf{T}}z \geq w^{\mathsf{T}}y\}\supseteq A$ is called a \emph{supporting halfspace} of $A$ at $y$.

Suppose that $A$ is a polyhedral closed convex set. The representation of $A$ as the intersection of finitely many halfspaces, that is, as
$A = \bigcap_{i = 1}^r\{y \in \mathbb{R}^q \mid (w^i)^{\mathsf{T}}y \geq a^i\}$
for some $r \in \mathbb{N}$, $w^i \in \mathbb{R}^q \setminus \{0\}$ and $a^i \in \mathbb{R}$, $i \in \{1, \dots, r\}$, is called \rev{an} \emph{$H$-representation} of $A$. Alternatively, $A$ is uniquely determined by a finite set $\{y^1, \dots, y^s\}\subseteq \R^q$ of vertices and a finite set $\{d^1, \dots, d^t\}\subseteq\R^q$ of directions via
$A = \conv \{y^1, \dots, y^s\} + \conv \cone \{d^1, \dots, d^t\},$
which is called \rev{a} \emph{$V$-representation} of $A$.


\section{Convex vector optimization} \label{sec:CVOP}

We consider a \emph{convex vector optimization problem (CVOP)} of the form
\begin{align}
\text{minimize } \Gamma(x) \text{ with respect to } \leq_C \text{ subject to } x \in \mathcal{X}, \tag{P}\label{eq:P}
\end{align}
where $C\subseteq \R^q$ is the \emph{ordering cone} of the problem, $\Gamma\colon X \to \mathbb{R}^q$ is the vector-valued \emph{objective function} defined on a convex set $X \subseteq \mathbb{R}^n$, and $\mathcal{X}\subseteq X$ is the \emph{feasible region}. The conditions we impose on $C, \Gamma, \mathcal{X}$ are stated in the next assumption.

\begin{assumption}\label{assmp}
	The following statements hold.
	\begin{enumerate}[(a)]
		\item $C$ is a closed convex cone that is also solid, pointed, and nontrivial. 
		\item $\Gamma$ is a $C$-convex and continuous function.
		\item $\mathcal{X}$ is a compact convex set with $\Int \mathcal{X}\neq \emptyset$.
	\end{enumerate}
\end{assumption}

The set
$\mathcal{P} := \cl (\Gamma(\mathcal{X}) + C)$ 
is called the \emph{upper image} of \eqref{eq:P}. Clearly, $\mathcal{P}$ is a closed convex set with $\mathcal{P}=\mathcal{P}+C$.

\begin{remark}\label{rem:compact}
	Note that, under Assumption \ref{assmp}, $\Gamma(\mathcal{X})$ is a compact set as the image of a compact set under a continuous function. Then, $\Gamma(\mathcal{X}) + C$ is a closed set as the algebraic sum of a compact set and a closed set \cite[Lemma~5.2]{guide2006infinite}. Hence, 
	we have
$	\mathcal{P}  = \Gamma(\mathcal{X}) + C$. 
\end{remark}

We recall the notion of boundedness for CVOP next.

\begin{definition}\label{defn:Boundedness}\cite[Definition 3.1]{lohne2014primal}
	\eqref{eq:P} is called bounded if $\mathcal{P} \subseteq \{y\} + C$ for some $y \in \mathbb{R}^q$.
\end{definition}

In view of \Cref{rem:compact}, it follows that \eqref{eq:P} is bounded under \Cref{assmp}.

The next definition recalls the relevant solution concepts for CVOP.

\begin{definition}\label{defn:minimizer}\cite[Definition 7.1]{heyde2011solution}
	A point $\bar{x} \in \mathcal{X}$ is said to be a \emph{(weak) minimizer} for \eqref{eq:P} if $\Gamma(\bar{x})$ is a (weakly) $C$-minimal element of $\Gamma(\mathcal{X})$. A nonempty set $\bar{\mathcal{X}} \subseteq \mathcal{X}$ is called an \emph{infimizer} of \eqref{eq:P} if $\cl \conv (\Gamma(\bar{\mathcal{X}}) + C) = \mathcal{P}$. An infimizer $\bar{\mathcal{X}}$ of \eqref{eq:P} is called a \emph{(weak) solution} of \eqref{eq:P} if it consists of only (weak) minimizers.
\end{definition}

In CVOP, it may be difficult or impossible to compute a solution in the sense of \Cref{defn:minimizer}, in general. Hence, we \rev{consider} the following notion of approximate solution.

\begin{definition}\rev{\cite[Definition 3.3]{vertexselection}}\label{defn:finite epsilon-solution}
	Suppose that \eqref{eq:P} is bounded and let $\epsilon>0$. Let $\bar{\mathcal{X}} \subseteq \mathcal{X}$ be a nonempty finite set and define $\mathcal{\bar{P}} :=\conv \Gamma(\bar{\mathcal{X}}) + C $. The set $\bar{\mathcal{X}}$ is called a \emph{finite $\epsilon$-infimizer} of \eqref{eq:P} if
	$\mathcal{\bar{P}} + \mathbb{B}_{\epsilon}  \supseteq \mathcal{P}$. 
	The set $\bar{\mathcal{X}}$ is called a \emph{finite (weak) $\epsilon$-solution} of \eqref{eq:P} if it is an $\epsilon$-infimizer that consists of only (weak) minimizers.
\end{definition}

For a finite (weak) $\epsilon$-solution $\bar{\mathcal{X}}$, it is immediate from \Cref{defn:finite epsilon-solution} that
\begin{equation}\label{eq:innerouter}
\mathcal{\bar{P}} + \mathbb{B}_{\epsilon} \supseteq \mathcal{P} \supseteq  \bar{\mathcal{P}}. 
\end{equation}
Hence, $\bar{\mathcal{X}}$ provides an inner and an outer approximation for the upper image $\mathcal{P}$.

\begin{remark}\label{rem:Hausdorff} 
	\rev{In \cite[Definition 3.3]{vertexselection} the statement of the definition is slightly different. Instead of $\mathcal{\bar{P}} + \mathbb{B}_{\epsilon} \supseteq \mathcal{P}$, the requirement is given as $\delta^H(\mathcal{P}, \mathcal{\bar{P}}) \leq \epsilon$. } 
\rev{However, both yield equivalent definitions.} Indeed, by \eqref{eq:innerouter} we have $\mathcal{P} \subseteq \mathcal{\bar{P}} + \mathbb{B}_{\epsilon}$ as well as $\bar{\mathcal{P}} \subseteq \mathcal{P}\subseteq \mathcal{P} + \mathbb{B}_{ \epsilon}$. Then, 
	\rev{$\delta^H(\mathcal{P}, \mathcal{\bar{P}}) \leq \epsilon$} follows by \eqref{eq:Hausdorff}. 
	\rev{The converse holds similarly by \eqref{eq:Hausdorff}.} 
\end{remark}


Given $w \in C^+\setminus\{0\}$, the following convex program is the well-known \emph{weighted sum scalarization} of \eqref{eq:P}:
\begin{equation}\label{eq:P1(w)}
\text{minimize } w^\mathsf{T}\Gamma(x) \text{ subject to }x \in \mathcal{X} .  \tag{WS$(w)$} 
\end{equation}

The following proposition is a standard result in vector optimization, it formulates the connection between weighted sum scalarizations and weak minimizers.

\begin{proposition}\label{prop:jahn2009vector} \cite[Corollary 2.3]{jahn1984}
	Let $w \in C^+ \setminus \{0\}$. Then, every optimal solution of \eqref{eq:P1(w)} is a weak minimizer of \eqref{eq:P}.
\end{proposition}

For the new notion of approximate solution in \Cref{defn:finite epsilon-solution}, we prove an existence result.

\begin{proposition}\label{prop:solution exists}
	Suppose that Assumption \ref{assmp} holds. Then, there exists a solution of \eqref{eq:P}. Moreover, for every $\epsilon > 0$, there exists a finite $\epsilon$-solution of \eqref{eq:P}.
	
	\begin{proof}
		The existence of a solution $\bar{\mathcal{X}}$ of \eqref{eq:P} follows by \cite[Proposition 4.2]{lohne2014primal}.
		\rev{By \cite[Proposition 4.3]{lohne2014primal}, for every $\epsilon > 0$, there exists a finite $\epsilon$-solution of \eqref{eq:P} in the sense of \cite[Definition 3.3]{lohne2014primal}. By \cite[Remark 3.4]{vertexselection}, an $\epsilon$-solution in the sense of \cite[Definition 3.3]{lohne2014primal} is also an $\epsilon$-solution in the sense of \Cref{defn:finite epsilon-solution}. Hence, the result follows.}
	\end{proof}
	
\end{proposition}


\section{Norm-minimizing scalarization} \label{sec:scal}

In this section, we describe the norm-minimizing scalarization model that we use in our proposed algorithm and provide some analytical results regarding this scalarization.

Let us fix an arbitrary norm $\norm{\cdot}$ on $\R^q$ and a point $v\in\R^q$. We consider the \emph{norm-minimizing scalarization} of \eqref{eq:P} given by
\begin{equation}\label{eq:P2(v)}
\text{minimize } \norm{z} \text{ subject to } \Gamma(x) - z - v \leq_C 0,\ x \in \mathcal{X},\ z \in \R^q.\tag{P$(v)$} 
\end{equation}
Note that this is a convex program.


\begin{remark}\label{rem:dist}
	The optimal value of \eqref{eq:P2(v)} is equal to $d(v,\mathcal{P})$, the distance of $v$ to the upper image $\mathcal{P}$. Indeed, by \Cref{rem:compact}, we have
	\begin{align}\notag 
	d(v,\mathcal{P})&=\inf_{y\in\mathcal{P}}\norm{v-y}=\inf\{\norm{z}\mid v+z\in\mathcal{P},\ z\in \R^q\}\\
	&=\inf\cb{\norm{z}\mid v+z\in \{\Gamma(x)\}+C,\ x\in\mathcal{X},\ z\in\R^q}\\
	&=\inf\cb{\norm{z}\mid \Gamma(x)-v-z\leq_C 0,\ x\in\mathcal{X},\ z\in\R^q}.
	\end{align}
\end{remark} 

In order to derive the Lagrangian dual of \eqref{eq:P2(v)}, we first pass to an equivalent formulation of \eqref{eq:P2(v)}. To that end, let us define a scalar function $f\colon X \times \R^q \to \overline{\R}$ and a set-valued function $G : X \times \mathbb{R}^q \rightrightarrows \mathbb{R}^q$ by
\[
f(x, z) := \norm{z} + I_{\mathcal{X}}(x),\quad G(x, z) := \{ \Gamma(x) - z - v \},\quad x\in X, z\in\R^q.
\]
Note that \eqref{eq:P2(v)} is equivalent to the following problem:
\begin{equation*} \label{eq:Pprime}
\text{minimize } f(x, z) \text{ subject to }   G(x, z) \cap -C \neq \emptyset,\  (x, z) \in X \times \mathbb{R}^q.      \tag{P$'(v)$}
\end{equation*}

To use the results from \cite[Section~3.3.1]{lohne2011vector} and \cite[Section~8.3.2]{khan2016set} for convex programming with set-valued constraints, we define the Lagrangian $L \colon X \times \R^q \times \R^q \rightarrow \overline{\R}$ for \eqref{eq:Pprime} by
\begin{equation}\label{eq:lagrangian}
L(x,z,w) := f(x, z) + \inf_{ u \in G(x, z) +C} w^{\mathsf{T}}u,\quad (x,z,w)\in X \times \R^q \times \R^q.
\end{equation}
Then, the dual objective function $\phi \colon \R^q \rightarrow \overline{\R}$ is defined by 
\[
\phi(w):= \inf_{x \in X, z \in \R^q} L(x,z,w),\quad w\in \R^q.
\]
By the definitions of $f, G$ and using the fact that $\inf_{ c \in C}  w^{\mathsf{T}}c = - I_{C^+}(w)$ for every $w\in\R^q$, we obtain
\begin{align}\notag 
\phi(w) = \begin{cases}
\inf_{x \in \mathcal{X}, z \in \R^q} \of{\norm{z} + w^{\mathsf{T}}(\Gamma(x) - z - v) } & \text{if } w \in C^+, \\
- \infty & \text{otherwise}.
\end{cases}
\end{align} 
Finally, the dual problem of \eqref{eq:Pprime} is formulated as
\begin{align}\label{eq:D2(v)} 
\text{maximize } \phi(w) \text{ subject to }w\in\R^q.\tag{D$(v)$} 
\end{align}
Then, the optimal value of \eqref{eq:D2(v)} is given by
\begin{align} \label{eq:pre_dualpr}
\sup_{w \in \R^q} \phi(w) &= \sup_{w \in C^+} \of{ \inf_{x \in \mathcal{X}} w^{\mathsf{T}}\Gamma(x) - \sup_{z \in \mathbb{R}^q} \of{w^{\mathsf{T}}z - \norm{z}}  - w^{\mathsf{T}}v }\\
&=\sup \cb{ \inf_{x \in \mathcal{X}} w^{\mathsf{T}}\Gamma(x) - w^{\mathsf{T}}v \mid \norm{w}_\ast \leq 1,\ w \in C^+} \notag 
\end{align}
since the conjugate function of $\norm{\cdot}$ is the indicator function of the unit ball of the dual norm $\norm{\cdot}_\ast$; see, for instance, \cite[Example 3.26]{boyd2004convex}.

The next proposition shows the strong duality between \eqref{eq:P2(v)} and \eqref{eq:D2(v)}.

\begin{proposition}\label{prop:optsol}
	Under Assumption \ref{assmp}, for every $v \in \mathbb{R}^q$, there exist optimal solutions $(x^v, z^v)$ and $w^v$ of problems \eqref{eq:P2(v)} and \eqref{eq:D2(v)}, respectively, and the optimal values coincide. 
\end{proposition}	
\begin{proof}
	Let us fix some $\tilde{x} \in \mathcal{X}$ and define $\tilde{z} := \Gamma(\tilde{x}) - v$. Clearly, $(\tilde{x}, \tilde{z})$ is feasible for \eqref{eq:P2(v)}. We consider the following problem	with compact feasible region in $\R^{n+q}$:
	\begin{equation}\label{eq:P(v, z)}
	\text{minimize } \norm{z} \text{ subject to }  \Gamma(x) - z - v \leq_C 0,\ \norm{z} \leq \norm{\tilde{z}},\ x \in \mathcal{X},\ z\in\R^q. 
	\end{equation}
	An optimal solution $(x^*, z^*)$ for the problem in \eqref{eq:P(v, z)} exists by Weierstrass Theorem and $(x^*, z^*)$ is also optimal for \eqref{eq:P2(v)}. To show the existence of an optimal solution of \eqref{eq:D2(v)}, we show that the following constraint qualification in \cite{khan2016set, lohne2011vector} holds for \eqref{eq:P2(v)}:
	\begin{equation}\label{eq:CQset}
	G(\dom f) \cap -\Int C \neq \emptyset,
	\end{equation}
	where $\dom f :=\{(x,z)\in X\times\R^q\mid f(x,z)<+\infty\}$. Since $\Int\mathcal{X}\neq\emptyset$ and $\Int C\neq\emptyset$ by Assumption \ref{assmp}, we may fix $x^0 \in \Int \mathcal{X}$, $y^0\in \Gamma(x^0)+\Int C$ and define $z^0 := y^0 - v$. We have
	$v + z^0 - \Gamma(x^0) \in \Int C$, equivalently, $G(x^0, z^0) \subseteq -\Int C.$ As	$(x^0, z^0)\in \dom f= \mathcal{X} \times \mathbb{R}^q$, it follows that \eqref{eq:CQset} holds.
	Moreover, the set-valued map $G : X \times \mathbb{R}^q \rightrightarrows \R^q$ is $C$-convex \cite[Section 8.3.2]{khan2016set}, that is, 	
	\begin{equation}\label{setCconvex}
	\lambda G(x^1, z) + (1 - \lambda) G(x^2, z) \subseteq G(\lambda (x^1, z) + (1-\lambda)(x^2, z)) + C
	\end{equation} 
	for every $x^1, x^2 \in X$, $z\in\R^q$, $\lambda \in [0,1]$. Indeed, by the $C$-convexity of $\Gamma\colon X \rightarrow \R^q$, we have
	\begin{equation*}
	\lambda (\Gamma(x^1) - z - v) + (1 - \lambda) (\Gamma(x^2) - z - v) \in \Gamma(\lambda x^1 + (1-\lambda) x^2)  - z - v + C
	\end{equation*}
	for every $x^1, x^2 \in X$, $z\in\R^q$, and $\lambda \in [0,1]$, from which \eqref{setCconvex} follows. Finally, since $f \colon X \times \R^q \rightarrow \overline{\R}$ is also convex, by \cite[Theorem~3.19]{lohne2011vector}, we have strong duality and dual attainment.
\end{proof}	

\begin{notation}
	From now on, we fix an arbitrary optimal solution $(x^v, z^v)$ of \eqref{eq:P2(v)} and an arbitrary optimal solution $w^v$ of \eqref{eq:D2(v)}. Their existence is guaranteed by \Cref{prop:optsol}.
\end{notation}

\begin{remark}\label{rem:wss}
	Note that $(x^v, z^v, w^v)$ is a saddle point of the Lagrangian for \eqref{eq:P2(v)} given by \eqref{eq:lagrangian}; see \cite[Section 5.4.2]{boyd2004convex}. Hence, we have 
	\begin{align}\notag 
	\sup_{w \in \mathbb{R}^q} L(x^v, z^v, w) = L(x^v, z^v, w^v) = \inf_{x \in \mathcal{X}, z \in \mathbb{R}^q} L(x, z, w^v).
	\end{align}
	The second equality yields that $(w^v)^{\mathsf{T}}\Gamma(x^v) = \inf_{x \in \mathcal{X}} (w^v)^{\mathsf{T}}\Gamma(x)$.	Hence, $x^v$ is an optimal solution of \eqref{eq:P1(w)} for $w = w^v$.
\end{remark}

In the next lemma, we characterize the cases where $z^v = 0$.

\begin{lemma}\label{lem:vlemma}
	Suppose that Assumption \ref{assmp} holds. The following statements hold. (a) If $v\notin\mathcal{P}$, then $z^v\neq 0$ and $w^v\neq 0$. (b) If $v\in\bd\mathcal{P}$, then $z^v=0$. (c) If $v\in\Int\mathcal{P}$, then $z^v=0$ and $w^v=0$.
	In particular, $v\in\mathcal{P}$ if and only if $z^v=0$.
\end{lemma}

\begin{proof}
	To prove (a), suppose that $v \notin \mathcal{P}$. To get a contradiction, we assume that $z^v=0$. Since $(x^v, z^v)$ is feasible for \eqref{eq:P2(v)}, we have $v = v + z^v \in \{\Gamma(x^v)\} + C \subseteq \mathcal{P}$, contradicting the supposition. Hence, $z^v\neq 0$. Moreover, if we had $w^v=0$, then the optimal value of \eqref{eq:D2(v)} would be zero and strong duality would imply that $\norm{z^v} = 0$, that is, $z^v = 0$. Therefore, we must have $w^v\neq 0$.
	
	To prove (b) and (c), suppose that $v\in \mathcal{P}$. By \Cref{rem:compact}, there exists $x\in\mathcal{X}$ such that $\Gamma(x)\leq_C v$. Then, $(x,0)$ is feasible for \eqref{eq:P2(v)}. Hence, the optimal value of \eqref{eq:P2(v)} is zero so that $\norm{z^v}=0$, that is, $z^v=0$. Suppose that we further have $v\in\Int\mathcal{P}$. Let $\delta \in \Int C$ be such that $v -\delta \in \mathcal{P}$. By \Cref{rem:compact}, there exists $x^{\delta}\in\mathcal{X}$ such that $\Gamma(x^{\delta})\leq_C v-\delta$, which implies $(w^v)^{\mathsf{T}}\Gamma(x^{\delta})\leq (w^v)^{\mathsf{T}}v-(w^v)^{\mathsf{T}}\delta$. Moreover, by strong duality, $\inf_{x\in \mathcal{X}} (w^v)^{\mathsf{T}}(\Gamma(x)-v)=0$ holds. Combining these gives
$	0=\inf_{x\in \mathcal{X}} (w^v)^{\mathsf{T}}(\Gamma(x)-v)\leq (w^v)^{\mathsf{T}}(\Gamma(x^\delta)-v)\leq -(w^v)^{\mathsf{T}}\delta\leq 0 $
	so that $ (w^v)^{\mathsf{T}} \delta = 0$. As $\delta\in \Int C$ and $w^v \in C^+$, we must have $w^v = 0$.
\end{proof}

The next proposition shows that solving \eqref{eq:P2(v)} when $v\notin \Int \mathcal{P}$ yields a weak minimizer for problem \eqref{eq:P}. 

\begin{proposition}\label{prop:zv0}
	Suppose that \Cref{assmp} holds. If $v \notin \Int \mathcal{P}$, then $x^v$ is a weak minimizer of \eqref{eq:P}, and $y^v := v + z^v \in \wMin_C (\mathcal{P})$.
\end{proposition}

\begin{proof}
	As $\mathcal{X}$ is nonempty and compact, we have $\mathcal{P} \neq \emptyset$ and $\mathcal{P} \neq \mathbb{R}^q$. By \cite[Definition~1.45 and Corollary~1.48~(iv)]{lohne2011vector}, we have $\wMin_C(\mathcal{P}) = \bd \mathcal{P}$. First, suppose that $v \in \bd \mathcal{P}$. Then, $z^v = 0$ by \Cref{lem:vlemma}. Together with primal feasibility, this implies 
	$\Gamma(x^v) \leq_C v$. As $v \in \wMin_C(\mathcal{P})$, by the definition of weakly $C$-minimal element, we have $\Gamma(x^v) \in \wMin_C(\mathcal{P})$. Hence, $x^v$ is a weak minimizer of \eqref{eq:P} in this case. Next, suppose that $v \notin \mathcal{P}$. Then, $w^v \neq 0$ by \Cref{lem:vlemma}. 
	By \Cref{rem:wss}, $x^v$ is an optimal solution of \eqref{eq:P1(w)} for $w = w^v \in C^+ \setminus \{0\}$. Hence, by \Cref{prop:jahn2009vector}, $x^v$ is a weak minimizer of \eqref{eq:P}. 
	
	Since $(x^v,z^v)$ is feasible for \eqref{eq:P2(v)}, $y^v \in \mathcal{P}$ holds. To get a contradiction, assume that $y^v \notin \wMin_C (\mathcal{P})$; hence, $y^v = v + z^v  \in \Int \mathcal{P}$. Then, there exists $\epsilon > 0$ such that
	$v + z^v - \epsilon\frac{z^v}{\norm{z^v}} \in \mathcal{P}$, which implies the existence of $\bar{x} \in \mathcal{X}$ with $ v + z^v - \epsilon\frac{z^v}{\norm{z^v}}  \in \{\Gamma(\bar{x})\} + C$. Let $\bar{z} := (\norm{z^v} - \epsilon) \frac{z^v}{\norm{z^v}}$. Then, $(\bar{x}, \bar{z})$ is feasible for \eqref{eq:P2(v)}. This is a contradiction as $\norm{\bar{z}} < \norm{z^v}$.
\end{proof}		

The following result shows that a supporting hyperplane of $\mathcal{P}$ at $y^v = v +z^v$ can be found by using a dual optimal solution $w^v$. 

\begin{proposition}\label{prop:supp_halfspace}
	Suppose that Assumption \ref{assmp} holds and $w^v \neq 0$. Then, the halfspace 
	$$\mathcal{H} = \{y \in \mathbb{R}^q \mid (w^v)^{\mathsf{T}}y \geq (w^v)^{\mathsf{T}} \Gamma(x^v) \}$$ contains the upper image $\mathcal{P}$. Moreover, $\bd \mathcal{H}$ is a supporting hyperplane of $\mathcal{P}$ both at $\Gamma(x^v)$ and $y^v = v +z^v$. In particular, $(w^v)^{\mathsf{T}} \Gamma(x^v) = (w^v)^{\mathsf{T}}y^v$.
\end{proposition}

\begin{proof}
	We clearly have $\Gamma(x^v) \in \bd \mathcal{P} \cap \mathcal{H}$ and $y^v \in \mathcal{P}\cap \mathcal{H}$. 
	Let $y \in \mathcal{P}$ be arbitrary and $\bar{x}\in \mathcal{X}$ be such that $\Gamma(\bar{x}) \leq_C y$. Consider the problems (P$(y)$) and (D$(y)$). Clearly, $(\bar{x}, 0)$ is feasible for (P$(y)$). Moreover, the optimal solution $w^v$ of \eqref{eq:D2(v)} is feasible for (D$(y)$). Using weak duality for (P$(y)$) and (D$(y)$), we obtain $0 \geq \inf_{x \in \mathcal{X}} (w^v)^{\mathsf{T}}\Gamma(x) - (w^v)^{\mathsf{T}}y$. Moreover, from strong duality for \eqref{eq:P2(v)} and \eqref{eq:D2(v)}, we have $\norm{z^v} = \inf_{x \in \mathcal{X}} (w^v)^{\mathsf{T}}\Gamma(x) - (w^v)^{\mathsf{T}}v.$ Hence,
	\[
	(w^v)^{\mathsf{T}}y \geq  \inf_{x \in \mathcal{X}} (w^v)^{\mathsf{T}}\Gamma(x) = \norm{z^v} + (w^v)^{\mathsf{T}}v.
	\]
	Note that $\norm{z^v} \geq (w^v)^{\mathsf{T}}z^v$ holds as $\norm{w^v}_* \leq  1$ by dual feasibility. Then, we obtain 
$	(w^v)^{\mathsf{T}}y  \geq (w^v)^{\mathsf{T}} y^v$. 
	In particular, we have $(w^v)^{\mathsf{T}}\Gamma(x^v) \geq (w^v)^{\mathsf{T}} y^v$ as $\Gamma(x^v) \in \mathcal{P}$. On the other hand, since $\Gamma (x^v) \leq_C y^v$ and $w^v \in C^+$, we also have $(w^v)^{\mathsf{T}}\Gamma (x^v) \leq (w^v)^{\mathsf{T}}y^v$. The equality $(w^v)^{\mathsf{T}}y^v= (w^v)^{\mathsf{T}} \Gamma(x^v)$ completes the proof as it implies $y\in \mathcal{H}$ (hence $\mathcal{P} \subseteq \mathcal{H}$) as well as $y^v \in \bd \mathcal{H}$.	
\end{proof}


\Cref{prop:supp_halfspace} provides a method to generate a supporting halfspace of $\mathcal{P}$ at $\Gamma(x^v)$ in which one uses an arbitrary dual optimal solution $w^v$. The next result shows that if the norm in \eqref{eq:P2(v)} is taken as the $\ell^p$-norm for some $p\in[1,+\infty)$, e.g., the Euclidean norm, then it is possible to generate a supporting halfspace to $\mathcal{P}$ at $\Gamma(x^v)$ using $z^v$ instead of $w^v$.

\begin{corollary}\label{prop:lpnorm}
	Suppose that \Cref{assmp} holds and $\norm{\cdot}=\norm{\cdot}_p$ for some $p\in [1,+\infty)$. Assume that $v \notin \mathcal{P}$. Then, the halfspace 
	\[
	\mathcal{H} = \Big\{y \in \mathbb{R}^q \mid \sum_{i = 1}^{q}\sgn(z_i^v)\lvert z_i^v\lvert^{p-1}y_i \geq \sum_{i = 1}^{q}\sgn(z_i^v)\lvert z_i^v\lvert^{p-1} \Gamma_i(x^v)\Big\}
	\]
	contains the upper image $\mathcal{P}$, where $\sgn$ is the usual sign function. Moreover, $\bd \mathcal{H}$ is a supporting hyperplane of
	$\mathcal{P}$ both at $\Gamma(x^v)$ and $y^v = v +z^v$.
\end{corollary}

\begin{proof}
	Consider \eqref{eq:P2(v)} and its Lagrange dual \eqref{eq:D2(v)}. Let us define $g(z):=\norm{z}_p-(w^v)^{\mathsf{T}}z$, $z\in\R^q$. The arbitrarily fixed dual optimal solution $w^v$ satisfies the first order condition with respect to $z$, that is,
	$	0 \in  \partial g(z^v)$. By the chain rule for subdifferentials, this is equivalent to
	\begin{align}\label{eq:subdifnorm}
	w^v \in \Big(\sum_{i=1}^q\lvert z^v_i\rvert^p\Big)^{\frac{1-p}{p}} \big(\lvert z^v_1\rvert^{p - 1}S_1 \times \dots \times \lvert z^v_q\rvert^{p - 1}S_q\big),
	\end{align}
	where, for each $i\in \{1,\ldots,q\}$, $S_i$ denotes the subdifferential of the absolute value function at $z^v_i$.
	Let $i\in\{1,\ldots,q\}$. Note that if $z^v_i \neq 0$, then we have $S_i= \{\sgn(z^v_i)\}$. On the other hand, if $z^v_i = 0$, then for each $s_i \in S_i$, we have $\lvert z^v_i\rvert^{p-1} s_i = 0$. Hence, by \eqref{eq:subdifnorm},
$	w^v_i = \Big(\sum_{i=1}^q\lvert z^v_i\rvert^p\Big)^{\frac{1-p}{p}} \lvert z_i^v\rvert^{p - 1}\sgn(z_i^v)$.
	The assertion follows from \Cref{prop:supp_halfspace}. 
\end{proof}


\section{The algorithm}\label{sec:alg}

We propose an outer approximation algorithm for finding a finite weak $\epsilon$-solution to CVOP as in \Cref{defn:finite epsilon-solution}. The algorithm is based on solving norm minimization scalarizations iteratively. The design of the algorithm is similar to the ``Benson-type algorithms" in the literature; see, for instance, \cite{benson1998outer, ehrgott2011approximation, lohne2014primal}. It starts by finding a polyhedral outer approximation $\mathcal{P}^{\out}_0$ of $\mathcal{P}$ and iterates in order to form a sequence $\mathcal{P}^{\out}_0 \supseteq \mathcal{P}^{\out}_1 \supseteq \ldots \supseteq \mathcal{P}$ of finer approximating sets.

Before providing the details of the algorithm, we impose a further assumption on $C$.
\begin{assumption} \label{assmp:C_poly}
	The ordering cone $C$ is polyhedral. 
\end{assumption}

\Cref{assmp:C_poly} implies that the dual cone $C^+$ is polyhedral. We denote the set of generating vectors of $C^+$ by $\{w^1,\ldots,w^J\}$, where $J\in\N$, i.e., $C^+ = \conv \cone \{w^1,\ldots,w^J\}$. 
Moreover, under \Cref{assmp}, $C^+$ is solid since $C$ is pointed. 
Hence, $J \geq q$.

The algorithm starts by solving the weighted sum scalarizations (WS$(w^1)$), \ldots, (WS$(w^J)$). For each $j\in \{1,\ldots, J\}$, the existence of an optimal solution $x^j \in \mathcal{X}$ of (WS$(w^j)$) is guaranteed by \Cref{assmp} (b, c). The initial set of weak minimizers is set as ${\mathcal{X}}_0 := \{x^1,\ldots,x^J\}$, see \Cref{prop:jahn2009vector}.\footnote{\rev{Alternatively, one may start with $\mathcal{X}_0=\emptyset$ in line 2 of \Cref{alg}. This would decrease $|\mathcal{X}|$ by $J$, the number of generating vectors $C$.}} The set $\mathcal{V}^{\text{known}}$, which keeps the set of all points $v\in\R^q$ for which \eqref{eq:P2(v)} and \eqref{eq:D2(v)} are solved throughout the algorithm, is initialized as the empty set. Moreover, similar to the primal algorithm in \cite{lohne2014primal}, the initial outer approximation is set as 
\begin{align}\label{eq:P_0}
\mathcal{P}^{\out}_0 := \bigcap_{j=1}^J \{y \in \mathbb{R}^q \mid (w^j)^{\mathsf{T}}y \geq (w^j)^{\mathsf{T}} \Gamma(x^j) \} 
\end{align}
(see lines 1-3 of \Cref{alg}). It is not difficult to see that $\mathcal{P}^{\out}_0 \supseteq \mathcal{P}$. Indeed, for each $\bar{y}\in \mathcal{P}$, there exists $\bar{x}\in \mathcal{X}$ such that $\Gamma(\bar{x})\leq_C \bar{y}$. Then, for each $j\in \{1,\ldots, J\}$, we have $(w^j)^{\mathsf{T}}\Gamma(\bar{x}) \leq (w^j)^{\mathsf{T}} \bar{y}$ which implies $(w^j)^{\mathsf{T}} \bar{y}\geq \inf_{x\in \mathcal{X}}(w^j)^{\mathsf{T}} \Gamma(x) = (w^j)^{\mathsf{T}} \Gamma(x^j)$ so that $\bar{y}\in\mathcal{P}^{\out}_0$. Moreover, as $C$ is pointed and \eqref{eq:P} is bounded, $\mathcal{P}^{\out}_0$ has at least one vertex,  see~\cite[Corollary 18.5.3]{rockafellar1970convex} (as well as \cite[Section 4.1]{lohne2014primal}).

At an arbitrary iteration $k\geq 0$ of the algorithm, the set $\mathcal{V}_k$ of vertices of the current outer approximation $\mathcal{P}^{\out}_{k}$ is computed first (line 6).\footnote{This is done by solving a vertex enumeration problem for $\mathcal{P}^{\out}_k$, that is, from the $H$-representation of $\mathcal{P}^{\out}_k$, its $V$-representation is computed. For the computational tests of~\Cref{sec:experiments}, we use \emph{bensolve tools} for this purpose \cite{lohne2017vector}.} Then, for each $v\in\mathcal{V}_{k}$, if not done before, the norm-minimizing scalarization \eqref{eq:P2(v)} and its dual \eqref{eq:D2(v)} are solved in order to find optimal solutions ($x^v$,$z^v$) and $w^v$, respectively \rev{(see \Cref{prop:optsol})}.\footnote{Note that many solvers yield both primal and dual optimal solutions when called only for one of the problems.} Moreover, $v$ is added to $\mathcal{V}^{\text{known}}$ (lines 7-10). If the distance $d(v,\mathcal{P})=\norm{z^v}$ is less than or equal to the predetermined approximation error $\epsilon > 0$, then $x^v$ is added to the set of weak minimizers \rev{(see \Cref{prop:zv0})}\footnote{Since the solution $x^v$ found in line 9 of \Cref{alg} is a weak minimizer, it is also possible to update the set of weak minimizers right after line 9 (without checking the value of $\norm{z^v}$) \rev{and subsequently ignore lines 13 and 18}. This would yield a finite weak $\epsilon$-solution with an increased cardinality.} and the algorithm continues by considering the remaining vertices of $\mathcal{P}_{k}$ (line 18). Otherwise, the supporting halfspace \begin{equation}\label{eq:Hk}
\mathcal{H}_k:=\{y \in \mathbb{R}^q \mid (w^v)^{\mathsf{T}}y \geq (w^v)^{\mathsf{T}} \Gamma(x^v) \}\end{equation}
of $\mathcal{P}$ at $\Gamma(x^v)$ is found (see \Cref{prop:supp_halfspace}); and the current approximation is updated as $\mathcal{P}^\out_{k+1} = \mathcal{P}^\out_{k} \cap \mathcal{H}_k$ (lines 12). The algorithm terminates if all the vertices in $\mathcal{V}_k$ are within $\epsilon$ distance to the upper image (lines 5, 15, 16, 22).

\begin{algorithm}[th]
	\caption{Outer Approximation Algorithm for \eqref{eq:P}}	
	\algsetup{
		linenosize=\small,
		linenodelimiter=.
	}
	\label{alg}
	\begin{algorithmic}[1]
		\STATE Compute an optimal solution $x^j$ of (WS$(w^j)$) for each $j \in \{1,\ldots,J\}$;
		\STATE Set $k = 0, \mathcal{{X}}_0 = \{x^1,\dots,x^J\}, \mathcal{V}^{\text{known}} = \emptyset$;
		\STATE Store an $H$-representation of $\mathcal{P}^{\out}_0$ according to \eqref{eq:P_0};
		\REPEAT
		\STATE Stop $\gets \TRUE$;
		\STATE Compute the set $\mathcal{V}_k$ of vertices of $\mathcal{P}^\out_k$ from its $H$-representation;
		\FOR{$v \in \mathcal{V}_k$}
		\IF{$v \notin \mathcal{V}^{\text{known}}$}
		\STATE Solve \eqref{eq:P2(v)} and \eqref{eq:D2(v)} to compute $(x^v,z^v)$ and $w^v$;
		\STATE $\mathcal{V}^{\text{known}} \gets \mathcal{V}^{\text{known}} \cup \{v\}$;
		\IF{$\norm{z^v} > \epsilon$}
		\STATE $\mathcal{P}^\out_{k+1} = \mathcal{P}^\out_k \cap \mathcal{H}_k$;
		\STATE $\mathcal{X}_{k+1} = \mathcal{X}_k$;
		\STATE $k \leftarrow k + 1$;
		\STATE Stop $\gets \FALSE$;
		\STATE break; 
		\ELSE
		\STATE ${\mathcal{X}}_{k} \gets {\mathcal{X}}_{k}\cup \{x^v\}$; 
		\ENDIF
		\ENDIF
		\ENDFOR
		\UNTIL{Stop}
		\RETURN		 
		$ \begin{cases} 
		\mathcal{{X}}_k &: \text{A finite weak $\epsilon$-solution to \eqref{eq:P}}; \\
		\mathcal{P}^\out_k &: \text{An outer approximation of~} \mathcal{P}.\\
		\end{cases}$
	\end{algorithmic}
\end{algorithm}

By the design of the algorithm, for each iteration $k\geq 0$, the set $\mathcal{P}^\out_k \supseteq \mathcal{P}$ is an outer approximation of the upper image; similarly, we define an inner approximation of $\mathcal{P}$ by
\begin{align} \label{eq:innerPk}
\mathcal{P}^{\inn}_k := \conv \Gamma(\mathcal{{X}}_k) + C \subseteq \mathcal{P}. 
\end{align}
Now, we present two lemmas regarding these inner and outer approximations. The first one shows that, for each $k\geq 0$, the sets $\mathcal{P}^\out_k$ and $\mathcal{P}^\inn_k$ have the same recession cone, which is the ordering cone $C$. With the second lemma, we see that, in order to compute the Hausdorff distance between $\mathcal{P}^\out_k$ and $\mathcal{P}$ (or $\mathcal{P}^\inn_k$), it is sufficient to consider the vertices $\mathcal{V}_k$ of $\mathcal{P}^{\out}_k$.

\begin{lemma}\label{lem:recession_directions}
	Suppose that Assumptions~\ref{assmp} and~\ref{assmp:C_poly} hold. Let $k \geq 0$. Then,
	\begin{align}\notag 
	\recc\mathcal{P}^\out_k = \recc\mathcal{P}^\inn_k = \recc\mathcal{P} = C.
	\end{align}
\end{lemma}	

\begin{proof}
	As $\conv\Gamma(\mathcal{X}_k)$ is a compact set, we have $\recc \mathcal{P}^\inn_k = C$ directly from \eqref{eq:innerPk}. Similarly, since $\Gamma(\mathcal{X})$ is compact by \Cref{assmp} and $\mathcal{P} = \Gamma(\mathcal{X})+C$ by \Cref{rem:compact}, we have $\recc \mathcal{P} = C$. 
	Since $\mathcal{P}^\out_k \supseteq \mathcal{P}$, we have $\recc\mathcal{P}^\out_k \supseteq \recc \mathcal{P} = C$; see \cite[Proposition~2.5]{dinh1989theory}. In order to conclude that $\recc\mathcal{P}^\out_k = C$, it is enough to show that $\recc\mathcal{P}^\out_0 \subseteq C$. Indeed, we have $\mathcal{P}^\out_k \subseteq \mathcal{P}^\out_0$, which implies that $\recc\mathcal{P}^\out_k \subseteq \recc\mathcal{P}^\out_0$. To prove $\recc \mathcal{P}^\out_0 \subseteq C$, let $\bar{y} \in \recc \mathcal{P}^\out_0$. Then, for each $y \in \mathcal{P}^\out_0$, we have $y + \bar{y} \in \mathcal{P}^\out_0$. By the definition of $\mathcal{P}^\out_0$ in \eqref{eq:P_0}, we have $(w^j)^{\mathsf{T}}(y + \bar{y}) \geq (w^j)^{\mathsf{T}}\Gamma(x^j)$ for each $j \in \{1, \dots, J\}$. In particular, as $\Gamma(x^j) \in \mathcal{P} \subseteq \mathcal{P}^\out_0$, we have
$	(w^j)^{\mathsf{T}}(\Gamma(x^j) + \bar{y}) \geq (w^j)^{\mathsf{T}}\Gamma(x^j)$,
	hence $(w^j)^{\mathsf{T}}\bar{y} \geq 0$ for each $j \in \{1, \dots, J\}$. By the definition of dual cone and using $(C^+)^+ = C$, we have $\bar{y} \in C$. The assertion holds as $\bar{y}\in \recc\mathcal{P}^\out_0$ is arbitrary. 
\end{proof}

\begin{lemma} \label{lem:Hdist_vertex}
	Suppose that Assumptions~\ref{assmp} and~\ref{assmp:C_poly} hold. Let $k\geq 0$. Then, 
	\begin{align}\notag 
	\delta^H(\mathcal{P}^\out_k, \mathcal{P}^\inn_k) = \max_{v \in \mathcal{V}_k} \ d(v, \mathcal{P}^\inn_k) \text{~and~~} \delta^H(\mathcal{P}^\out_k, \mathcal{P}) = \max_{v \in \mathcal{V}_k} \ d(v, \mathcal{P}),
	\end{align}
	where $\mathcal{V}_k$ is the set of vertices of $\mathcal{P}^\out_k$. 
\end{lemma}
\begin{proof}
	To see the first equality, note that
	\begin{align}\notag 
	\delta^H(\mathcal{P}^\out_k, \mathcal{P}^\inn_k) = \max \bigg\{\sup_{y \in \mathcal{P}^\out_k} \ d(y, \mathcal{P}^\inn_k), \sup_{y \in \mathcal{P}^\inn_k} \ d(y, \mathcal{P}^\out_k) \bigg\} = \sup_{y \in \mathcal{P}^\out_k} \ d(y, \mathcal{P}^\inn_k)
	\end{align}
	as $\mathcal{P}^\inn_k \subseteq \mathcal{P}^\out_k$. Moreover, since $\recc \mathcal{P}^\inn_k = \recc\mathcal{P}^\out_k$ by \Cref{lem:recession_directions}, $\delta^H (\mathcal{P}^\out_k, \mathcal{P}^\inn_k) < \infty$ holds correct; see \cite[Lemma~6.3.15]{facchinei2007finite}. Since $\mathcal{P}^\out_k$ is a polyhedron with at least one vertex, $d(\cdot , \mathcal{P}^\inn_k)$ is a convex function 
	(see, for instance, \cite[Proposition 1.77]{mordukhovich2013easy}) and  $\delta^H(\mathcal{P}^\out_k, \mathcal{P}^\inn_k)<+\infty$, we have
	\begin{align}\notag 
	\sup_{y \in \mathcal{P}^\out_k}\ d(y, \mathcal{P}^\inn_k) = \max_{v\in \mathcal{V}_k} \ d(v, \mathcal{P}^\inn_k)
	\end{align}
	from \cite[Propositions~7-8]{lohne2017solving}. The second equality can be shown similarly by noting that $\delta^H(\mathcal{P}^\out_k, \mathcal{P}) \leq \delta^H(\mathcal{P}^\out_k, \mathcal{P}^\inn_k) < \infty$ thanks to $\mathcal{P}^\inn_k \subseteq \mathcal{P} \subseteq \mathcal{P}^\out_k$.
\end{proof}

\begin{theorem}\label{thm:alg}
	Under Assumptions~\ref{assmp} and~\ref{assmp:C_poly}, \Cref{alg} works correctly: if the algorithm terminates, then it returns a finite weak $\epsilon$-solution to \eqref{eq:P}.
\end{theorem}

\begin{proof}
	For all $j\in \{1,\ldots, J\}$, an optimal solution $x^j$ of (WS$(w^j)$) exists since $\mathcal{X}$ is compact and $x \mapsto (w^j)^{\mathsf{T}}\Gamma(x)$ is continuous by the continuity of $\Gamma \colon X \rightarrow \R^q$ provided by \Cref{assmp}. Moreover, $x^j$ is a weak minimizer of $\eqref{eq:P}$ by \Cref{prop:jahn2009vector}. Thus, $\mathcal{{X}}_0$ consists of weak minimizers.	
	
	Since $\eqref{eq:P}$ is a bounded problem and $C$ is a pointed cone, the set $\mathcal{P}^\out_0$ contains no lines. Hence,  
	$\mathcal{P}^\out_0$ has at least one vertex \cite[Corollary 18.5.3]{rockafellar1970convex}, that is, $\mathcal{V}_0 \neq \emptyset$. Moreover, as detailed after equation \eqref{eq:P_0}, $\mathcal{P}^\out_0\supseteq \mathcal{P}$, hence we have $v\notin \Int \mathcal{P}$ for any $v \in\mathcal{V}_0$. As it will be discussed below, for $k\geq 1$, $\mathcal{P}^\out_k$ is constructed by intersecting $\mathcal{P}^\out_0$ with supporting halfspaces of $\mathcal{P}$. This implies $\mathcal{V}_k\neq \emptyset$ and $\mathcal{V}_k \subseteq \R^q \setminus \Int \mathcal{P}$ hold for any $k\geq 0$.
	
	By \Cref{prop:optsol}, optimal solutions ($x^v$,$z^v$) and $w^v$ to \eqref{eq:P2(v)} and \eqref{eq:D2(v)}, respectively, exist. Moreover, by \Cref{prop:zv0}, $x^v$ is a weak minimizer of $\eqref{eq:P}$. 
	If $\norm{z^v}>0$, then $v\notin \mathcal{P}$, hence $w^v \neq 0$ by \Cref{lem:vlemma}. By \Cref{prop:supp_halfspace}, $\mathcal{H}_k$ given by \eqref{eq:Hk} is a supporting halfspace of $\mathcal{P}$ at $\Gamma(x^v)$. Then, since $\mathcal{P}^\out_0\supseteq \mathcal{P}$ and $\mathcal{P}^\out_{k+1} = \mathcal{P}^\out_k\cap\mathcal{H}_k$, we have $\mathcal{P}^\out_k \supseteq \mathcal{P}$ for all $k\geq 0$.
	
	Assume that the algorithm stops after $\hat{k}$ iterations. Since $\mathcal{X}_{\hat{k}}$ is finite and consists of weak minimizers, to prove $\mathcal{X}_{\hat{k}}$ is a finite weak $\epsilon$-solution of $\eqref{eq:P}$ as in \Cref{defn:finite epsilon-solution}, it is sufficient to show that
$	{\mathcal{P}}^\inn_{\hat{k}} + \mathbb{B}_\epsilon \supseteq \mathcal{P}$,
	where ${\mathcal{P}}^\inn_{\hat{k}} = \conv \Gamma(\mathcal{X}_{\hat{k}}) + C$.	
	
	Note that by the stopping condition, we have $\norm{z^v} \leq \epsilon$, hence $x^v \in \mathcal{X}_{\hat{k}}$, for all $v\in \mathcal{V}_{\hat{k}}$. Moreover, since $(x^v, z^v)$ is feasible for \eqref{eq:P2(v)}, 
$	v + z^v \in \{\Gamma(x^v)\} + C \subseteq \conv \Gamma(\mathcal{X}_{\hat{k}}) + C = {\mathcal{P}}^\inn_{\hat{k}}$.
	Then, by \Cref{lem:Hdist_vertex},
	\begin{align}\notag 
	\delta^H({\mathcal{P}}^\out_{\hat{k}},\mathcal{P}^\inn_{\hat{k}}) = \max_{v \in \mathcal{V}_{\hat{k}}} \ d(v, \mathcal{{P}}^\inn_{\hat{k}}) = \max_{v \in \mathcal{V}_{\hat{k}}} \inf_{u \in \mathcal{{P}}^\inn_{\hat{k}}} \norm{u - v} \leq \max_{v \in \mathcal{V}_{\hat{k}}} \norm{z^v} \leq \epsilon.
	\end{align}
	Consequently, ${\mathcal{P}}^\inn_{\hat{k}} + \mathbb{B}_\epsilon \supseteq \mathcal{P}$ follows since
$	\mathcal{{P}}^\inn_{\hat{k}} + \mathbb{B}_\epsilon \supseteq \mathcal{P}^\out_{\hat{k}} \supseteq \mathcal{P}$.
\end{proof}


\section{The modified algorithm} \label{sec:alg2}

In this section, we propose a modification of \Cref{alg} and prove its correctness. Recall that, in \Cref{thm:alg}, we show that \Cref{alg} returns a finite weak $\epsilon$-solution, provided that it terminates. The purpose of this section is to propose an algorithm for which we can prove finiteness as well; the proof of finiteness will be presented separately in \Cref{sec:fin}.

The main feature of the modified algorithm, \Cref{alg1} is that, in each iteration, it intersects the current outer approximation of $\mathcal{P}$ with a fixed halspace $S$ and considers only the vertices of the intersection. As we describe next, the halfspace $S$ is formed in such a way that $\mathcal{P}\cap S$ is compact and $\Gamma(\mathcal{X})\subseteq \mathcal{P}\cap S$. 

\setcounter{algo}{2}
\begin{algorithm}
	\caption{Modified Outer Approximation Algorithm for \eqref{eq:P}}
	\label{alg1}
	\algsetup{
		linenosize=\small,
		linenodelimiter=.
	}
	\begin{algorithmic}[1]
		\STATE Compute an optimal solution $x^j$ of (WS$(w^j)$) for each $j \in \{1,\ldots,J\}$;
		\STATE Set $k = 0, \bar{\mathcal{{X}}}_0 = \{x^1,\dots,x^J\}, \mathcal{V}^{\text{known}} = \emptyset$;
		\STATE Store an $H$-representation of $\mathcal{P}^{\out}_0$ according to \eqref{eq:P_0};
		\STATE Compute the set $\mathcal{V}_0$ of vertices of $\mathcal{P}^\out_0$ from its $H$-representation;
		\STATE $\bar{\mathcal{P}}^\out_0 = \mathcal{P}^\out_0$;
		\FOR{$v\in\mathcal{V}_0$}
		\STATE Solve \eqref{eq:P2(v)} and \eqref{eq:D2(v)} to compute $(x^v, z^v)$, $w^v$, and $d(v,\mathcal{P})$;
		\STATE $\mathcal{V}^{\text{known}} \gets \mathcal{V}^{\text{known}} \cup \{v\}$;
		\IF{$\norm{z^v} > \epsilon$}
		\STATE $\bar{\mathcal{P}}^\out_{0} \gets \bar{\mathcal{P}}^\out_0 \cap \{y \in \mathbb{R}^q \mid (w^v)^{\mathsf{T}}y \geq (w^v)^{\mathsf{T}} \Gamma(x^v) \}$;
		\ELSE
		\STATE $\bar{\mathcal{X}_{0} }\gets \bar{\mathcal{X}}_0\cup \{x^v\}$;
		\ENDIF
		\ENDFOR
		\STATE Compute $\beta$ by \Cref{rem:S_computation} and $\alpha$ by \eqref{eq:alpha};
		\STATE Store an $H$-representations $S$ according to \eqref{eq:S};
		\REPEAT
		\STATE Stop $\gets \TRUE$;
		\STATE Compute the set $\bar{\mathcal{V}}_k$ of vertices of $\bar{\mathcal{P}}^\out_k\cap S$ from its $H$-representation;
		\FOR{$v \in \bar{\mathcal{V}}_k$} 
		\STATE Follow lines 8-20 of \Cref{alg} using $\bar{\mathcal{P}}^{\out}_{k}$, $\bar{\mathcal{P}}^{\out}_{k+1}$, $\bar{\mathcal{X}_{k}}$, $\bar{\mathcal{X}}_{k+1}$\footnotemark;
		\ENDFOR
		\UNTIL{Stop}
		\RETURN		 
		$ \begin{cases} 
		\bar{\mathcal{{X}}}_k &: \text{A finite weak $\epsilon$-solution to \eqref{eq:P}}; \\
		\bar{\mathcal{P}}^\out_k &: \text{An outer approximation of~} \mathcal{P}.\\
		\end{cases}$
	\end{algorithmic}
\end{algorithm}
\footnotetext{\rev{More precisely, in \Cref{alg}, we apply $\bar{\mathcal{P}}^{\out}_{k+1}=\bar{\mathcal{P}}^{\out}_{k}\cap\mathcal{H}_k$ in line 12, $\bar{\mathcal{X}}_{k+1}=\bar{\mathcal{X}}_k$ in line 13, $\bar{\mathcal{X}}_k\gets \bar{\mathcal{X}}\cup\{x^v\}$ in line 18.}}


For the construction of $S$, let us define
\begin{equation}\label{eq:w_bar}
\bar{w} := \frac{\sum_{j = 1}^{J} w^j}{\norm{\sum_{j = 1}^{J} w^j}_\ast} \in \Int C^+,
\end{equation}
and fix $\beta\in \R$ such that
\begin{equation} \label{eq:beta}
\beta \geq \sup_{x \in \mathcal{X}} \bar{w}^{\mathsf{T}} \Gamma(x).
\end{equation}
The existence of such $\beta$ is guaranteed by \Cref{assmp} since $\mathcal{X}\subseteq \R^n$ is a compact set and $x \mapsto \bar{w}^{\mathsf{T}}\Gamma(x)$ is a continuous function on $X$ under this assumption. Note that $\beta$ is an upper bound on the optimal value of an optimization problem that may fail to be convex, in general. We address some possible ways of computing $\beta$ in \Cref{rem:S_computation} below.

\begin{remark}\label{rem:S_computation}
	Note that, $\bar{w}^{\mathsf{T}} \Gamma(x)$ is convex as $\Gamma:X\to\R^q$ is $C$-convex and $\bar{w} \in \Int C^+$. Hence, $\sup_{x \in \mathcal{X}} \bar{w}^{\mathsf{T}} \Gamma(x)$ is a concave minimization problem over a compact convex set $\mathcal{X}$. Concave minimization is a well-known problem type in optimization for which numerous algorithms available in the literature to find a global optimal solution, see for instance \cite{benson1995concave, benson1991branch}. In our case, it is enough to run a single iteration of one such algorithm to find $\beta$.
	
\end{remark}

In addition to $\bar{w}$ and $\beta$, we fix $\alpha\in\R$ such that
\begin{align}\label{eq:alpha}
\alpha > \max_{{v \in \mathcal{V}_0}} (\bar{w}^{\mathsf{T}} v - \beta)^+ + \delta^H(\mathcal{P}^{\out}_0, \mathcal{P}),
\end{align}
where $\mathcal{P}^{\out}_0$ is the initial outer approximation used in \Cref{alg}, $\mathcal{V}_0$ is the set of vertices of $\mathcal{P}^\out_0$, and $a^+:=\max\{a,0\}$ for $a\in\R$. By \Cref{lem:Hdist_vertex}, we have $\delta^H(\mathcal{P}^\out_0, \mathcal{P}) = \max_{v \in \mathcal{V}_0} \ d(v, \mathcal{P})$. Moreover, for each $v\in\mathcal{V}_0$, we have $d(v,\mathcal{P}) = \norm{z^v}$, where $(x^v, z^v)$ is an optimal solution to \eqref{eq:P2(v)}. Hence, $\alpha$ can be computed once \eqref{eq:P2(v)} is solved for each $v\in\mathcal{V}_0$. Finally, using $\bar{w}, \alpha,\beta$, we define
\begin{align}\label{eq:S}
S := \{ y\in\R^q \mid \bar{w}^{\mathsf{T}}y \leq \beta+ \alpha\}. 
\end{align}

\Cref{alg1} starts with an initialization phase followed by a main loop that is similar to \Cref{alg}. The initialization phase starts by constructing $\mathcal{P}^\out_0$ according to \eqref{eq:P_0} (lines 1-4 of \Cref{alg1}) and computing the set $\mathcal{V}_0$ of its vertices. For each $v\in\mathcal{V}_0$, the problems \eqref{eq:P2(v)} and \eqref{eq:D2(v)} are solved (line 7). The common optimal value $\norm{z^v}=d(v,\mathcal{P})$ is used in the calculation of $\delta^H(\mathcal{P}^\out_0, \mathcal{P})$ as described above. Moreover, these problems yield a supporting halfspace of $\mathcal{P}$ which is used to refine the outer approximation (line 10) if $\norm{z^v}$ exceeds the predetermined error $\epsilon>0$. Otherwise, the solution of \eqref{eq:P2(v)} is added to the set of weak minimizers (line 12). We denote by $\bar{\mathcal{P}}_0^{\out}$ the refined outer approximation that is obtained at the end of the initialization phase.

The main loop of \Cref{alg1} (lines 17-23) follows the same structure as \Cref{alg} except that it computes the set $\bar{\mathcal{V}}_k$ of all vertices of $\bar{\mathcal{P}}^\out_k\cap S$ (as opposed to that of $\mathcal{P}^\out_k$) at each iteration $k\geq 0$ (line 19). The algorithm terminates if all the vertices in $\bar{\mathcal{V}}_k$ are within $\epsilon$ distance to $\mathcal{P}$. As opposed to \Cref{alg}, in \Cref{alg1}, the norm-minimizing scalarization \eqref{eq:P2(v)} is not solved for a vertex $v$ of $\bar{\mathcal{P}}^\out_k$ if it is not in $S$. In \Cref{thm:alg2}, we will prove that the modified algorithm works correctly even if it ignores such vertices. The next proposition, even though it is not directly used in the proof of \Cref{thm:alg2}, provides a geometric motivation for this result. In particular, it shows that if \eqref{eq:P2(v)} is solved for some $v \notin \Int S$, then the supporting halfspace obtained as in \Cref{prop:supp_halfspace} supports the upper image at a weakly $C$-minimal but not $C$-minimal element of the upper image.
\begin{proposition}\label{lem:yvk_wMin}
	Let $v$ be a vertex of $\bar{\mathcal{P}}^\out_k$ for some $k\geq 1$. If $v \notin \Int S$, then $y^v= v+z^v \in \wMin_C(\mathcal{P}) \setminus \Min_C(\mathcal{P})$.
\end{proposition}	
\begin{proof}
	Suppose that $v \notin \Int S$. As $v$ is a vertex of $\bar{\mathcal{P}}^\out_k$, we have $v\notin\Int\mathcal{P}$. By \Cref{prop:zv0}, $y^{v} \in \wMin_C(\mathcal{P})$; in particular, $y^v \in \mathcal{P}$. Using \Cref{rem:compact}, there exist $\tilde{x} \in \mathcal{X}, \tilde{c} \in C$ such that $y^v = \Gamma(\tilde{x}) + \tilde{c}$. Next, we show that $\tilde{c} \neq 0$, which implies $y^{v} \notin \Min_C(\mathcal{P})$. From H\"{o}lder's inequality and \eqref{eq:w_bar}, we have
$	\bar{w}^{\mathsf{T}} (v - y^{v}) \leq \norm{y^{v} - v} \norm{\bar{w}}_* = \norm{y^{v} - v} = \norm{z^{v}}$.
	Moreover, using $\bar{\mathcal{P}}^\out_k \subseteq \mathcal{P}^\out_0$, we obtain 
	\begin{align}\notag 
	\norm{z^{v}} =d(v,\mathcal{P})\leq \sup_{v^\prime\in \bar{\mathcal{P}}^\out_k}d(v^\prime,\mathcal{P})\leq \sup_{v^\prime\in \mathcal{P}^\out_0} d(v^\prime,\mathcal{P})= \delta^H(\mathcal{P}, \mathcal{P}^{\out}_0)  < \alpha, 
	\end{align}
	where the last inequality follows from \eqref{eq:alpha}. Together, these imply $\bar{w}^{\mathsf{T}} y^{v} > \bar{w}^{\mathsf{T}} v - \alpha$. Using \eqref{eq:S}, \eqref{eq:beta} and $v \notin \Int S$, we also have $\bar{w}^{\mathsf{T}}  v - \alpha \geq \sup_{x \in \mathcal{X}} \bar{w}^{\mathsf{T}}  \Gamma(x)$. Now, since $\bar{w}^{\mathsf{T}}  y^{v} > \sup_{x \in \mathcal{X}} \bar{w}^{\mathsf{T}}  \Gamma(x)$, it must be true that $y^{v} \notin \Gamma(\mathcal{X})$, which implies $\tilde{c} \neq 0$. 
\end{proof}

\rev{It may happen that a vertex of $\bar{\mathcal{P}}^\out_k$, $k\geq 1$, falls outside $S$. We will illustrate this case in \Cref{rem:outsideS} of \Cref{sec:experiments}.}
	
With the following lemma and remark, we show that $S$ satisfies the required properties mentioned at the beginning of this section. Note that \Cref{lem:PcapScompact} implies the compactness of $\mathcal{P}\cap S$ since $\mathcal{P}\subseteq \mathcal{P}^\out_0$.

\begin{lemma} \label{lem:PcapScompact}
	Suppose that Assumptions~\ref{assmp} and~\ref{assmp:C_poly} hold. Let $\mathcal{P}^\out_0$ and $S$ be as in \eqref{eq:P_0} and \eqref{eq:S}, respectively. Then, $\mathcal{P}^\out_0 \cap S$ is a compact set.
\end{lemma}	

\begin{proof}
	Since $\mathcal{P}^\out_0$ and $S$ are closed sets, $\mathcal{P}^\out_0 \cap S$ is closed. Let $r \in\recc(\mathcal{P}^\out_0 \cap S)$ and $a \in \mathcal{P}^\out_0 \cap S$ be arbitrary. For every $\lambda \geq 0$, we have $a + \lambda r \in \mathcal{P}^\out_0 \cap S$. By the definition of $\mathcal{P}^\out_0$, this implies
	\begin{align}\label{eq:inP0}
	(w^j)^{\mathsf{T}} (a + \lambda r) \geq (w^j)^{\mathsf{T}} \Gamma(x^j) = \inf_{x \in \mathcal{X}} \ (w^j)^{\mathsf{T}} \Gamma(x),
	\end{align}
	for each $j \in \{1, \dots, J\}$. On the other hand, by the definition of $S$ in \eqref{eq:S}, we have 
	\begin{align} \label{eq:inS}
	\bar{w}^{\mathsf{T}} (a + \lambda r) \leq \beta + \alpha.
	\end{align}
	Since \eqref{eq:inP0} and \eqref{eq:inS} hold for every $\lambda \geq 0$, we have $	(w^j)^{\mathsf{T}} r \geq 0 $ for each $j \in \{1, \dots, J\}$ and $\bar{w}^{\mathsf{T}} r = \sum_{j = 1}^{J} (w^j)^{\mathsf{T}} r / \norm{\sum_{j = 1}^{J} w^j}_\ast\leq 0$, respectively. Together, these imply that
	\begin{align}\label{eq:rperpwj}
	(w^j)^{\mathsf{T}} r = 0
	\end{align}
	for each $j \in \{1, \dots, J\}$.
	Recall that $J \geq q$ is implied by Assumptions~\ref{assmp} and~\ref{assmp:C_poly}. Consider the $q \times J$ matrix, say $W$, whose columns are the generating vectors of $C^+$. Since $C^+$ is solid, which follows from $C$ being pointed, $\rank W = q$, see for instance \cite[Theorem 3.1]{burns1974polyhedral}. Consider a $q\times q$ invertible submatrix $\tilde{W}$ of $W$. From \eqref{eq:rperpwj}, we have $\tilde{W}^{\mathsf{T}}r = 0 \in \R^q$, which implies $r = 0$. As $r\in \recc(\mathcal{P}^\out_0 \cap S)$ is chosen arbitrarily, $\mathcal{P}^\out_0 \cap S$ is bounded, hence compact.	
\end{proof}


\begin{remark}\label{rem:Scontains} 
	It is clear by the definition of $S$ that $\Gamma(\mathcal{X})\subseteq S$. Let $k\geq0$. Since $\mathcal{P}\subseteq \bar{\mathcal{P}}^\out_k$, we also have $\Gamma(\mathcal{X}) \subseteq \bar{\mathcal{P}}^\out_k \cap S$. Then, using \Cref{rem:compact}, we obtain $\mathcal{P} = \Gamma(\mathcal{X}) + C \subseteq (\bar{\mathcal{P}}^\out_k \cap S) + C$. \rev{Also note that, if the algorithm terminates, then all the vertices in $\bar{\mathcal{V}}_k$ are within $\epsilon$ distance to $\mathcal{P}$.}
\end{remark}

\begin{remark}\label{rem:Soutside}
	\rev{In line 19 of \Cref{alg1}, if we compute $\mathcal{V}_k \cap S$ instead of $\bar{\mathcal{V}}_k$, that is, if we just ignore the vertices of $\bar{\mathcal{P}}_k^{\out}$ which are outside $S$, then we cannot guarantee returning a finite weak $\epsilon$-solution. This is because there may exist some vertices of $\bar{\mathcal{P}}_k^{\out}$ that are out of $S$ with distance to the upper image being larger than $\epsilon$. 
	Moreover, 
	$\conv (\mathcal{V}_k \cap S) + C$ may not contain $\mathcal{P}$.
	This will be illustrated in \Cref{rem:outsideS} of \Cref{sec:experiments}.}
\end{remark}

\begin{theorem}\label{thm:alg2}
	Under Assumptions~\ref{assmp} and \ref{assmp:C_poly}, \Cref{alg1} works correctly: if the algorithm terminates, then it returns a finite weak $\epsilon$-solution to \eqref{eq:P}.
\end{theorem}

\begin{proof}
	Similar to the proof of \Cref{thm:alg}, the set $\mathcal{\bar{X}}_0$ is initialized by weak minimizers of $\eqref{eq:P}$, and $\mathcal{V}_0$ is nonempty. Moreover, for each $v\in \mathcal{V}_0$, optimal solutions $(x^v,z^v)$ and $w^v$ exist (\Cref{prop:optsol}); $x^v$ is a weak minimizer (\Cref{prop:zv0}) and $\{y\in \R^q \mid (w^v)^{\mathsf{T}}y \geq (w^v)^{\mathsf{T}} \Gamma(x^v)\}$ is a supporting halfspace of $\mathcal{P}$ at $\Gamma(x^v)$ (\Cref{prop:supp_halfspace}). Hence, $\bar{\mathcal{P}}^\out_0 \supseteq \mathcal{P}$ is an outer approximation and $\mathcal{\bar{X}}_0$ consists of weak minimizers. By the definition of $S$, we have $\mathcal{V}_0 \subseteq S$. Hence, the set $\bar{\mathcal{V}}_0$ is nonempty.
	
	
	
	
	
	Considering the main loop of the algorithm, we know by \Cref{prop:optsol} that optimal solutions ($x^v$,$z^v$) and $w^v$ to \eqref{eq:P2(v)} and \eqref{eq:D2(v)}, respectively, exist. Moreover, if $\norm{z^v}\geq \epsilon$, then $w^v \neq 0$ by \Cref{lem:vlemma}. Hence, by \Cref{prop:supp_halfspace}, $\mathcal{H}_k$ given by \eqref{eq:Hk} is a supporting halfspace of $\mathcal{P}$. This implies $\bar{\mathcal{P}}^\out_k \supseteq \mathcal{P}$ for all $k\geq 0$. Since $\bar{\mathcal{P}}^\out_k \supseteq \mathcal{P} \supseteq \Gamma(\mathcal{X})$ and $S \supseteq \Gamma(\mathcal{X})$, see \Cref{rem:Scontains}, the set $\bar{\mathcal{P}}^\out_k\cap S$ is nonempty. Moreover, as $\bar{\mathcal{P}}^\out_k \subseteq \mathcal{P}^\out_0$, it is true that $\bar{\mathcal{P}}^\out_k \cap S$ is compact by \Cref{lem:PcapScompact}. Then, $\bar{\mathcal{V}}_k$ is nonempty for all $k\geq 0$. Note that every vertex $v\in \bar{\mathcal{V}}_k$ satisfies $v \notin \Int \mathcal{P}$. Indeed, since $v$ is a vertex of $\bar{\mathcal{P}}^\out_k \cap S$, it must be true that $v\in\bd \mathcal{H}_{\bar{k}}$ for some $\bar{k}\leq k$. The assertion follows since $\bd \mathcal{H}_{\bar{k}}$ is a supporting hyperplane of $\mathcal{P}$. Then, by \Cref{prop:zv0}, $x^v$ is a weak minimizer of \eqref{eq:P}.
	
	Assume that the algorithm stops after $\hat{k}$ iterations. Clearly, $\bar{\mathcal{X}}_{\hat{k}}$ is finite and consists of weak minimizers. By \Cref{defn:finite epsilon-solution}, it remains to show that 
$	\bar{\mathcal{P}}^\inn_{\hat{k}} + \mathbb{B}_\epsilon \supseteq \mathcal{P}$
	holds, where $\bar{\mathcal{P}}^\inn_{\hat{k}} := \conv \Gamma(\bar{\mathcal{X}}_{\hat{k}}) + C$. By the stopping condition, for every $v\in \bar{\mathcal{V}}_{\hat{k}}$, we have $\norm{z^v} \leq \epsilon$, hence $x^v \in \bar{\mathcal{X}}_{\hat{k}}$. Moreover, since $(x^v, z^v)$ is feasible for \eqref{eq:P2(v)}, 	
	\begin{align}\label{eq:P-in}
	v + z^v \in \{\Gamma(x^v)\} + C \subseteq \conv \Gamma(\bar{\mathcal{X}}_{\hat{k}}) + C = \bar{\mathcal{P}}^\inn_{\hat{k}}.
	\end{align}
	By \Cref{rem:Scontains}, it is true that
	$\mathcal{\bar{P}}^\inn_{\hat{k}} \subseteq \mathcal{P} \subseteq (\bar{\mathcal{P}}^\out_{\hat{k}} \cap S) + C$. Moreover, as $\conv \Gamma(\bar{\mathcal{X}}_{\hat{k}})$ and $\bar{\mathcal{P}}^\out_{\hat{k}} \cap S$ are compact sets, $\recc \mathcal{\bar{P}}^\inn_{\hat{k}} = \recc ((\bar{\mathcal{P}}^\out_{\hat{k}} \cap S)+C) = C$. By repeating the arguments in the proof of \Cref{lem:Hdist_vertex}, it is easy to check that
	\begin{align}\notag 
	\delta^H(\mathcal{\bar{P}}^\inn_{\hat{k}}, (\bar{\mathcal{P}}^\out_{\hat{k}} \cap S) + C) = \max_{v \in \bar{\mathcal{V}}_{\hat{k}}^C} \ d(v, \mathcal{\bar{P}}^\inn_{\hat{k}}),
	\end{align}
	where $\bar{\mathcal{V}}_{\hat{k}}^C$ denotes the set of all vertices of $(\bar{\mathcal{P}}^\out_{\hat{k}} \cap S)+C$. Observe that every vertex $v$ of $(\bar{\mathcal{P}}^{\out}_{\hat{k}} \cap S) + C$ is also a vertex of $\bar{\mathcal{P}}^{\out}_{\hat{k}} \cap S$, that is, $\bar{\mathcal{V}}_{\hat{k}}^C \subseteq \bar{\mathcal{V}}_{\hat{k}}$. Then, we obtain
	\begin{align}\notag 
	\delta^H(\mathcal{\bar{P}}^\inn_{\hat{k}}, (\bar{\mathcal{P}}^\out_{\hat{k}} \cap S) + C) \leq \max_{v \in \bar{\mathcal{V}}_{\hat{k}}} \ d(v, \mathcal{\bar{P}}^\inn_{\hat{k}}) = \max_{v \in \bar{\mathcal{V}}_{\hat{k}}} \inf_{u \in \mathcal{\bar{P}}^\inn_{\hat{k}}} \norm{u - v} \leq \max_{v \in \bar{\mathcal{V}}_{\hat{k}}} \norm{z^v} \leq \epsilon,
	\end{align}
	where the penultimate inequality follows by \Cref{eq:P-in}. Since
	\begin{align} \label{eq:appr_err}
	\mathcal{\bar{P}}^\inn_{\hat{k}} + \mathbb{B}_\epsilon \supseteq (\bar{\mathcal{P}}^\out_{\hat{k}} \cap S) + C \supseteq \mathcal{P},
	\end{align}
$\bar{\mathcal{P}}^\inn_{\hat{k}} + \mathbb{B}_{\epsilon} \supseteq \mathcal{P}$ follows.
\end{proof}

\section{Finiteness of the modified algorithm}\label{sec:fin}

The correctness of Algorithms \ref{alg} and \ref{alg1} are proven in Theorems \ref{thm:alg} and \ref{thm:alg2}, respectively. In this section, we prove the finiteness of \Cref{alg1}. We provide two technical results before proceeding to the main theorem. 

\begin{lemma} \label{lem:Bset}
	Suppose that Assumptions~\ref{assmp} and~\ref{assmp:C_poly} hold. Let $v \notin \mathcal{P}$ and $\mathcal{H}$ be the halfspace defined by \Cref{prop:supp_halfspace}. If $\norm{z^{v}} \geq \epsilon$, then $B \cap \mathcal{H} = \emptyset$, where $B := \cb{y \in \{v\} + C \mid \norm{y - v} \leq \frac{\epsilon}{2} }$.
\end{lemma}
\begin{proof}
	Consider \eqref{eq:P2(v)} and its Lagrange dual \eqref{eq:D2(v)}.
	The arbitrarily fixed dual optimal solution $w^v$ satisfies the first order condition with respect to $z$, which can be expressed as $w^v \in  \partial \norm{z^v} $. Note that the subdifferential of $\norm{\cdot}$ at $z^v$ has the variational characterization
	\[
	\partial \norm{z^v} = \cb{w \in\R^q \mid  \sup\{(w^\prime)^{\mathsf{T}}z^v \mid \norm{w^\prime}_\ast\leq 1\}=w^{\mathsf{T}}z^v,\ \norm{w}_{\ast}\leq 1}, 
	\]
	which follows by applying \cite[Theorem 23.5]{rockafellar1970convex}. Since the dual norm of $\norm{\cdot}_{\ast}$ is $\norm{\cdot}$,
	\begin{align} \label{eq:subdif_gennorm}
	w^v\in \partial \norm{z^v} = \{w \in\R^q \mid \norm{z^v} = w^{\mathsf{T}}z^v , \ \norm{w}_{\ast} \leq 1 \}, 
	\end{align}
	Let $\bar{y} \in \mathcal{H}$ be arbitrary. From the definition of $\mathcal{H}$ and \eqref{eq:subdif_gennorm}, we have $(w^v)^{\mathsf{T}}\bar{y} \geq (w^v)^{\mathsf{T}}(v+z^v) = (w^v)^{\mathsf{T}}v + \norm{z^v}$. Equivalently,
$	(w^v)^{\mathsf{T}}(\bar{y}-v) \geq \norm{z^v}$.
	On the other hand, from H\"older's inequality and \eqref{eq:subdif_gennorm}, we have
$	\abs{(w^v)^{\mathsf{T}}(\bar{y}-v)} \leq \norm{w^v}_{\ast}\norm{\bar{y}-v} \leq \norm{\bar{y}-v}$.
	If $\norm{z^v}\geq\epsilon$, then from the last two inequalities, 
	we obtain
$	\norm{\bar{y}-v} \geq \abs{(w^v)^{\mathsf{T}}(\bar{y}-v)} \geq \norm{z^v} \geq \epsilon$.
	Therefore, $\bar{y} \notin B$, which implies $B \cap \mathcal{H} = \emptyset$.
\end{proof}


\begin{lemma}\label{lem:S2}
	Suppose that Assumptions~\ref{assmp} and~\ref{assmp:C_poly} hold. Let $k \geq 0$, $v$ be a vertex of $\bar{\mathcal{P}}^\out_k$, $S$ be as in \eqref{eq:S}; and define
	\begin{align}\notag 
	S_2 := \{ y\in \R^q \mid \bar{w}^{\mathsf{T}}y \leq \beta + 2\alpha\},
	\end{align}
	where $\bar{w}, \beta, \alpha$ are defined by \eqref{eq:w_bar}, \eqref{eq:beta}, \eqref{eq:alpha}, respectively. If $v \in S$, then $v+z^{v} \in \Int S_2$.
\end{lemma}		
\begin{proof}
	Let $\tilde{\mathcal{V}}_k$ denote the set of all vertices of $\bar{\mathcal{P}}^\out_k$. It is given that $v \in \tilde{\mathcal{V}}_k$. Using \Cref{rem:dist} and the arguments in the proof of \Cref{lem:Hdist_vertex}, we obtain 
$	\delta^H(\mathcal{P}, \bar{\mathcal{P}}^\out_k) = \max_{\tilde{v} \in \tilde{\mathcal{V}}_k} d(\tilde{v}, \mathcal{P}) \geq d(v, \mathcal{P}) =  \norm{z^v}$.
	From \eqref{eq:alpha} and the inclusion $\mathcal{P}^\out_0 \supseteq \bar{\mathcal{P}}^\out_k$, we have
	$\delta^H(\mathcal{P}, \bar{\mathcal{P}}^\out_k) \leq \delta^H(\mathcal{P}, \mathcal{P}^\out_0) < \alpha$, which implies $\norm{z^v}< \alpha$.
	Then, using H\"{o}lder's inequality together with $\norm{\bar{w}}_* = 1$, we obtain
$	\bar{w}^{\mathsf{T}} z^{v} \leq \norm{z^{v}} \norm{\bar{w}}_* = \norm{z^{v}} < \alpha$.
	On the other hand, $v \in S$ implies that $\bar{w}^{\mathsf{T}} v \leq \beta + \alpha$. Then, $v+z^{v} \in \Int S_2$ follows since $\bar{w}^{\mathsf{T}} (v+z^v) < \beta + 2\alpha$.
\end{proof}

\begin{theorem}\label{thm:finiteness}
	Suppose that Assumptions~\ref{assmp} and~\ref{assmp:C_poly} hold. \Cref{alg1} terminates after a finite number of iterations.
	\begin{proof}
		By the construction of the algorithm, the number of vertices of $\bar{\mathcal{P}}^\out_k$ is finite for every $k \geq 0$. It is sufficient to prove that there exists $K \geq$ 0 such that for every vertex $v \in 
		\bar{\mathcal{V}}_K$ of $\bar{\mathcal{P}}^\out_K \cap S$, we have $\norm{z^{v}} \leq \epsilon$. To get a contradiction, assume that for every $k \geq 0$, there exists a vertex $v^k \in \bar{\mathcal{V}}_k$ 
		such that $\norm{z^{v^k}} > \epsilon$. For convenience, an optimal solution $(x^{v^k}, z^{v^k})$ of (P($v^k$)) is denoted by $(x^k, z^{k})$ throughout the rest of the proof. Let $S_2$ be as in \Cref{lem:S2}. Then, following similar arguments presented in the proof of \Cref{lem:PcapScompact}, one can show that $\mathcal{P}^\out_0 \cap S_2$ is compact. 
		
		Let $k\geq 0$ be arbitrary. We define
		$B^k := \cb{y \in \{v^k\} + C \mid \norm{y - v^k} \leq \frac{\epsilon}{2}}$.
		Note that $B^k$ is a compact set in $\R^q$ and, as $C$ is solid by \Cref{assmp}, $B^k$ has a positive volume, which is free of the choice of $k$. Next, we show that $B^k \subseteq \mathcal{P}^\out_0 \cap S_2$. Repeating the arguments in the proof of \Cref{lem:recession_directions}, it can be shown that $\recc \bar{\mathcal{P}}^\out_k = C$. Then, since $v^k \in \bar{\mathcal{P}}^\out_k$, it holds
		\begin{equation} \label{eq:Bk_in_Pk}
		\{v^k\} + C \subseteq \bar{\mathcal{P}}^\out_k \subseteq \mathcal{P}^\out_0.
		\end{equation} 
		Hence, $B^k \subseteq \mathcal{P}^\out_0$. To see $B^k \subseteq S_2$, 
		let $y\in B^k$. From H\"{o}lder's inequality and \eqref{eq:w_bar},
		\begin{align}\label{eq:Holder2}
		\bar{w}^{\mathsf{T}} (y - v^k) \leq \norm{y - v^k} \norm{\bar{w}}_{*} = \norm{y - v^k} \leq \frac{\epsilon}{2}.
		\end{align}	
		As there exists $v^0 \in \bar{\mathcal{V}}_0$ with $\norm{z^0} > \epsilon$, it holds true that $\delta^H(\mathcal{P}, \mathcal{P}^\out_0) > \epsilon$. From \eqref{eq:alpha}, it follows that $\alpha > \epsilon$. Then, from \eqref{eq:Holder2}, we obtain $\bar{w}^{\mathsf{T}} (y - v^k) \leq \frac{\epsilon}{2} < \alpha$. Since $v^k \in S$, this implies $\bar{w}^{\mathsf{T}} y < \bar{w}^{\mathsf{T}} v^k + \alpha \leq \beta + 2\alpha$, hence $y \in S_2$. 
		
		Next, we prove that $B^i \cap B^j = \emptyset$ for every $ i, j \geq$ 0 with $i \neq j $. Assume without loss of generality that $i < j $. Thus, $\bar{\mathcal{P}}^\out_j \subseteq \bar{\mathcal{P}}^\out_{i+1}$. From \Cref{lem:Bset}, we have $B^i \cap \mathcal{H}_i = \emptyset$, where $\mathcal{H}_i$ is the supporting halfspace at $\Gamma(x^i)$ as obtained in \Cref{prop:supp_halfspace}. 
		This implies $B^i \cap \bar{\mathcal{P}}^\out_{j} = \emptyset$ as we have $\bar{\mathcal{P}}^\out_j  \subseteq \bar{\mathcal{P}}^\out_{i+1} = \bar{\mathcal{P}}^\out_{i} \cap \mathcal{H}_{i} \subseteq \mathcal{H}_{i}$. 
		On the other hand, we have $B^j \subseteq \bar{\mathcal{P}}^\out_j$ from \eqref{eq:Bk_in_Pk}. Thus, $B^i \cap B^j = \emptyset$. These imply that there is an infinite number of non-overlapping sets, having strictly positive fixed volume, contained in a compact set $\mathcal{P}^\out_0 \cap S_2$, a contradiction.
	\end{proof}
\end{theorem}

We conclude this section with a convergence result regarding the Hausdorff distance between the upper image and its polyhedral approximations.

\begin{corollary}\label{cor:limit_1}
	Suppose that Assumptions~\ref{assmp} and~\ref{assmp:C_poly} hold. \rev{Let $\epsilon=0$ and \Cref{alg1} be modified by introducing a cutting order based on selecting a farthest away vertex (instead of an arbitrary vertex) in line 20. Then,}
	\[
	\lim_{k\to \infty} \delta^H(\bar{\mathcal{P}}^\inn_k, \mathcal{P}) = \lim_{k\to \infty} \delta^H((\bar{\mathcal{P}}^\out_k \cap S)+C, \mathcal{P}) = 0,
	\]
	where $\bar{\mathcal{P}}^\inn_{k} := \conv \Gamma(\bar{\mathcal{X}}_{k}) + C$, the sets $\bar{\mathcal{X}}_{k}, \bar{\mathcal{P}}^\out_k$ are as described in \Cref{alg1}, and $S$ is given by \eqref{eq:S}.
\end{corollary}
\begin{proof}
	\rev{Note that \Cref{alg1} is finite by \Cref{thm:finiteness} for an arbitrary vertex selection rule, hence, also when a farthest away vertex is selected in line 20. Therefore, given $\varepsilon > 0$, there exists $K(\varepsilon) \in \mathbb{N}$ such that the set $\bar{\mathcal{X}}_{K(\varepsilon)}$ is a finite weak $\varepsilon$-solution as in \Cref{defn:finite epsilon-solution} and $\delta^H(\bar{\mathcal{P}}^\inn_{K(\varepsilon)}, \mathcal{P}) \leq \varepsilon$ by \Cref{rem:Hausdorff}. 
	Let us consider the modified algorithm (with $\epsilon=0$). If $\delta^H(\bar{\mathcal{P}}^\inn_{k}, \mathcal{P})=0$ for some $k\in\N$, then it is clear that $\lim_{k\to \infty} \delta^H(\bar{\mathcal{P}}^\inn_k, \mathcal{P}) = 0$. Suppose that $\delta^H(\bar{\mathcal{P}}^\inn_{k}, \mathcal{P})>0$ for every $k\in\N$. With the farthest away vertex selection rule, when we run the algorithm with $\epsilon=0$ and $\epsilon=\varepsilon>0$, the two work in the same way until the one with $\epsilon=\varepsilon$ stops. Hence, they find the same inner approximation $\bar{\mathcal{P}}^\inn_{K(\varepsilon)}$ at step $K(\varepsilon)$. Let $k_0:=K(\varepsilon)$. By an induction argument, for every $n\in\N$, the inequality $\delta^H(\bar{\mathcal{P}}^\inn_{k_{n}}, \mathcal{P}) \leq \frac{\varepsilon}{n}$ will be satisfied by the algorithm with $\epsilon=0$ for some $k_n>k_{n-1}$. Hence, $\lim_{n\rightarrow\infty}\delta^H(\bar{\mathcal{P}}^\inn_{k_{n}}, \mathcal{P})=0$, which implies that $\lim_{k\to \infty} \delta^H(\bar{\mathcal{P}}^\inn_k, \mathcal{P}) = 0$ by the monotonicity of $(\delta^H(\bar{\mathcal{P}}^\inn_k, \mathcal{P}) )_{k\in\N}$.
%
}
	
	
	Moreover, similar to the discussion in the proof of \Cref{thm:alg2}, see \eqref{eq:appr_err}, it can be shown that 
$\delta^H((\bar{\mathcal{P}}^\out_k \cap S)+C, \mathcal{P}) \leq \delta^H(\bar{\mathcal{P}}^\inn_k, \mathcal{P})$ holds for each $k\in\N$. \rev{Hence, $\lim_{k\to \infty} \delta^H((\bar{\mathcal{P}}^\out_k \cap S)+C, \mathcal{P}) = 0$.}
\end{proof}

\begin{remark}\label{rem:limit_1}
	\rev{In \Cref{cor:limit_1}, choosing the farthest away vertex in each iteration is critical. Indeed, without this rule, the algorithm run with $\epsilon=0$ and $\epsilon=\varepsilon>0$ may not work in the same way due to line 11 in \Cref{alg1}. In such a case, we may not use \Cref{thm:finiteness} to argue that the algorithm with $\epsilon=0$ satisfies $\delta^H(\bar{\mathcal{P}}^\inn_{k}, \mathcal{P}) \leq \varepsilon$ for some $k\in\N$. Indeed, it might happen in case of non-polyhedral $\mathcal{P}$ that the algorithm keeps updating the outer approximation by focusing only on one part of $\mathcal{P}$. Thus, the Hausdorff distance at the limit may not be zero.} 
	%
		%
	%
	%
\end{remark}

\section{Examples and computational results}
\label{sec:experiments}

In this section, we examine few numerical examples to evaluate the performance of Algorithms \ref{alg} and \ref{alg1} in comparison with the primal algorithm (referred to as Algorithm 3 here) in \cite{lohne2014primal}. The algorithms are implemented using MATLAB R2018a along with CVX, a package to solve convex programs \cite{cvx, gb08}, and \emph{bensolve tools} \cite{lohne2017vector} to solve the scalarization and vertex enumeration problems in each iteration, respectively. The tests are performed using a 3.6 GHz Intel Core i7 computer with a \rev{64 GB RAM}. 

We consider three examples: \Cref{ex:1} is a standard illustrative example with a linear objective function, see \cite{ehrgott2011approximation,lohne2014primal}, in which both the feasible region and its image are the Euclidean unit ball centered at the vector $e=(1,\ldots,1)^{\mathsf{T}}\in \R^q$. In \Cref{ex:2}, the objective functions are nonlinear while the constraints are linear; in \Cref{ex:3}, nonlinear terms appear both in the objective function and constraints \cite[Examples 5.8, 5.10]{ehrgott2011approximation}, \cite{miettinen2006experiments}. 



\begin{example} \label{ex:1} We consider the following problem for $q\in \{2,3,4\}$, where  $C=\R_+^q$:
	\begin{align}\notag 
	\text{minimize~~} & \Gamma(x) = x \text{~~with respect to} \leq_{C} \\
	\text{subject to~~} &	\norm{x - e}_2 \leq 1, \:\: x\in\R^q.
	\end{align}
\end{example}

\begin{example} \label{ex:2} Let $a^1=(1,1)^\mathsf{T},a^2=(2,3)^\mathsf{T}, a^3=(4,2)^\mathsf{T}$. Consider
	\begin{align}\notag 
	\text{minimize~~} \ & \Gamma(x) = (\norm{x-a^1}_2^2,\norm{x-a^2}_2^2,\norm{x-a^3}_2^2)^\mathsf{T}
	\text{~~with respect to} \leq_{\R_+^3} \\
	\text{subject to~~} \ & x_1 + 2x_2 \leq 10, \:\: 0 \leq x_1 \leq 10, \:\: 0 \leq x_2 \leq 4, \:\: x\in\R^2.
	\end{align}
\end{example}

\begin{example} \label{ex:3} \rev{Let $\hat{b}^1=(0,10,120), \hat{b}^2=(80, -448, 80), \hat{b}^3=(-448,80,80)$ and $b^1,b^2,b^3\in\R^n$.} Consider
	\begin{align}\notag 
	\text{minimize~~} & \Gamma(x) = ( \norm{x}_2^2+b^1x,  \norm{x}_2^2+b^2x, \norm{x}_2^2+b^3x)^\mathsf{T}
	\text{~~with respect to} \leq_{\R_+^3} \\
	\text{subject to~~} & \norm{x}_2^2 \leq 100, \:\: 0 \leq x_i \leq 10 \text{~for~} i\in\rev{\{1,\dots,n\}}, \:\: x\in\R^n.
	\end{align}
\begin{enumerate}[(a)]
	\item \rev{Let $n=3$ and $b^1=\hat{b}^1, b^2=\hat{b}^2, b^3=\hat{b}^3$.}
	
	\item \rev{Let $n=9$ and $b^1=(\hat{b}^1,\hat{b}^1,\hat{b}^1), b^2=(\hat{b}^2,\hat{b}^2,\hat{b}^2), b^3=(\hat{b}^3,\hat{b}^3,\hat{b}^3)$.}
\end{enumerate}
\end{example}


We solve these examples with Algorithms \ref{alg} and \ref{alg1}, where the norm in \eqref{eq:P2(v)} is taken as the $\ell_p$ norm for $p\in\{1,2,\infty\}$. An outer approximation of the upper image for each example is shown in \Cref{fig:ex}. For Algorithm 3, the fixed direction vector for the scalarization model is taken as $\frac{e}{\norm{e}_p} \in \R^q$, again for $p\in\{1,2,\infty\}$. This way, it is guaranteed that when Algorithm 3 returns a finite weak $\epsilon$-solution in the sense of \cite[Definition 3.3]{lohne2014primal}, this solution is also a finite weak $\epsilon$-solution in the sense of \Cref{defn:finite epsilon-solution} for the corresponding norm-ball, \rev{see \cite[Remark 3.4]{vertexselection}. We solve Examples \ref{ex:2} and \ref{ex:3} for the approximation errors as chosen by Ehrgott et al. in \cite{ehrgott2011approximation} for the same examples.} 

The computational results are presented in Tables~\ref{tab:ex1}-\ref{tab:ex3}, which show the approximation error ($\epsilon$), the algorithm (Alg), the cardinality of finite weak $\epsilon$-solution ($\abs{\bar{\mathcal{X}}}$), the number of optimization problems (Opt), and the number of vertex enumeration problems (En) solved through the algorithm, along with the respective times taken to solve these problems (T$_\text{opt}$, T$_\text{en}$), as well as the total runtime of the algorithm (T), \rev{where T$_\text{opt}$, T$_\text{en}$ and T are in seconds}. 

\begin{figure}[tbhp]
	\centering
	\subfloat[\Cref{ex:1} with $\epsilon=0.01$]{\label{fig:ex1}\includegraphics[width=2.1in]{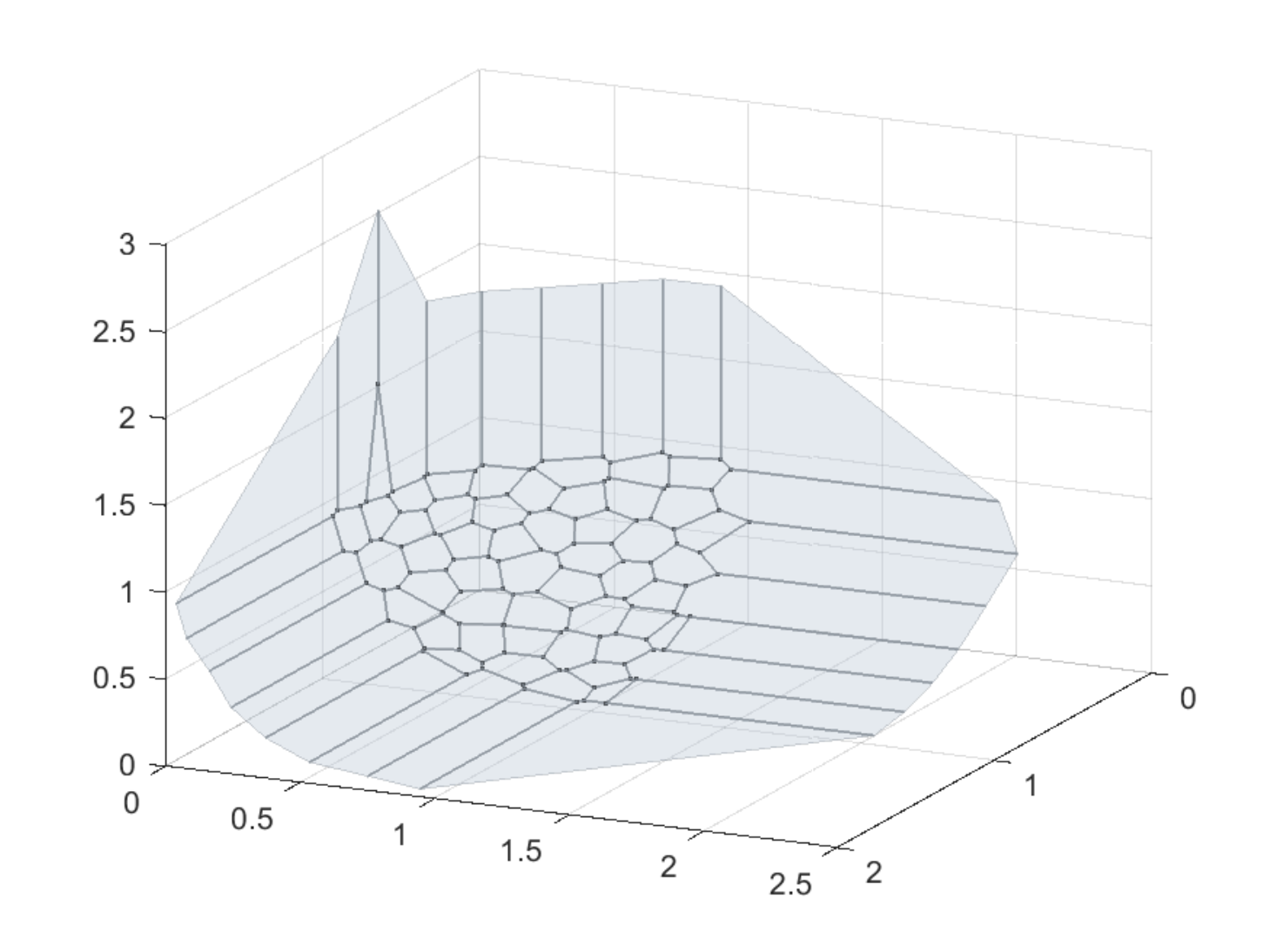}}
	\subfloat[\Cref{ex:2} with $\epsilon=0.05$]{\label{fig:ex2}\includegraphics[width=2.1in]{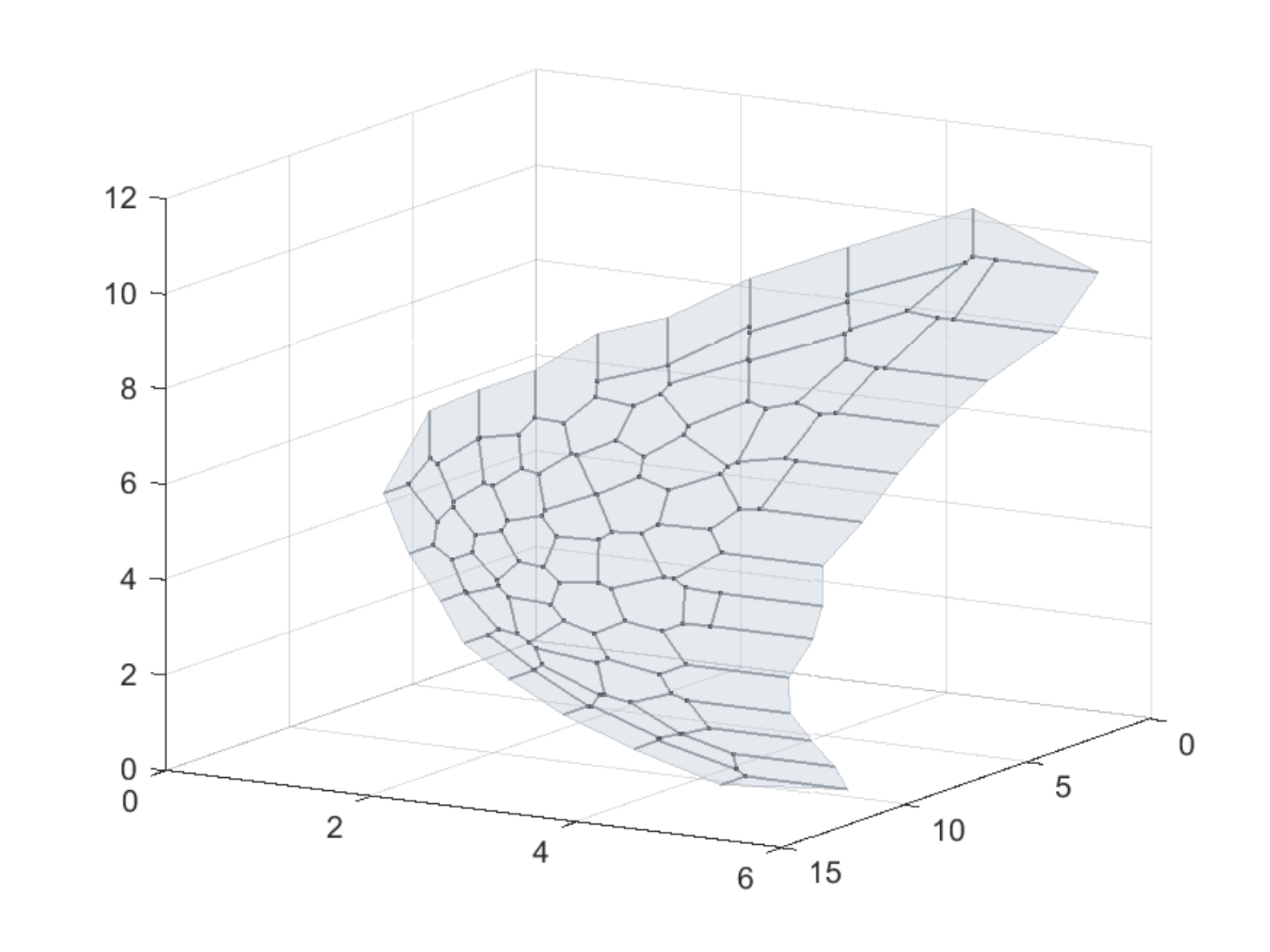}}
	\subfloat[\Cref{ex:3} with $\epsilon=5$]{\label{fig:ex3}\includegraphics[width=2.1in]{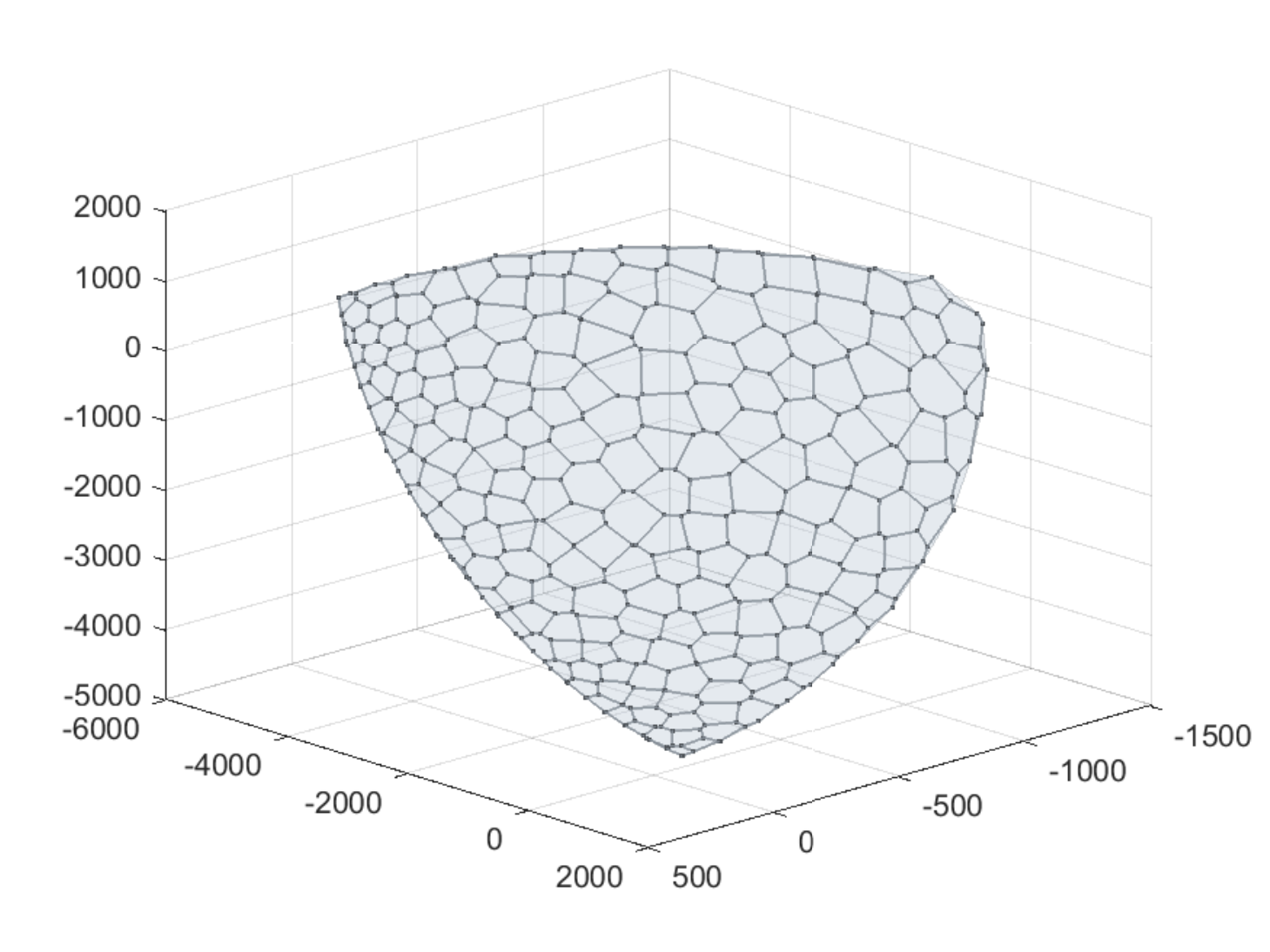}}
	\caption{Outer approximations obtained from \Cref{alg} using $\ell_2$ norm}
	\label{fig:ex}
\end{figure}

\rev{Note that we could not run the algorithms in a few settings. We cannot solve \Cref{ex:1}, $q=4$ for $\epsilon=0.1$ by Algorithms \ref{alg} and \ref{alg1} when $p=2$, and by Algorithms \ref{alg1} and 3 when $p=\infty$. Similarly, we cannot solve \Cref{ex:2} by Algorithm 3 for any $p \in \{1, 2, \infty\}$. Moreover, we cannot solve this example by \Cref{alg} when $p=\infty$ for both $\epsilon$ values and when $p=1$ for $\epsilon=0.01$. Since it is not possible to provide a comparison, we do not report the results for $p=\infty$ in \Cref{tab:ex2}. Finally, \Cref{ex:3} cannot be solved by Algorithms~\ref{alg} and \ref{alg1} when $p=1$. Hence, \Cref{tab:ex3} does not show the results for $p=1$ for any setting. The main reason that the algorithms cannot solve these instances is the limitations of \emph{bensolve tools} in vertex enumeration. 
}

\begin{table}[htbp]
	\centering
	\caption{\rev{Computational results for \Cref{ex:1}} }
	\label{tab:ex1}%
	\resizebox{\textwidth}{!}{
		\begin{tabular}{|c|c|c|cccccc|c|cccccc|}
			\cline{3-16}    \multicolumn{1}{r}{} &     & \multicolumn{7}{c|}{$q$=3}                & \multicolumn{7}{c|}{$q$=4} \\
			\hline
			\multirow{2}[2]{*}{$p$} & \multirow{2}[2]{*}{Alg} & \multicolumn{1}{c}{\multirow{2}[2]{*}{$\epsilon$}} & \multirow{2}[2]{*}{$\abs{\bar{\mathcal{X}}}$} & \multirow{2}[2]{*}{Opt} & \multirow{2}[2]{*}{T$_\text{opt}$} & \multirow{2}[2]{*}{En} & \multirow{2}[2]{*}{T$_\text{en}$} & \multirow{2}[2]{*}{T} & \multicolumn{1}{c}{\multirow{2}[2]{*}{$\epsilon$}} & \multirow{2}[2]{*}{$\abs{\bar{\mathcal{X}}}$} & \multirow{2}[2]{*}{Opt} & \multirow{2}[2]{*}{T$_\text{opt}$} & \multirow{2}[2]{*}{En} & \multirow{2}[2]{*}{T$_\text{en}$} & \multirow{2}[2]{*}{T} \\
			&     & \multicolumn{1}{c}{} &     &     &     &     &     &     & \multicolumn{1}{c}{} &     &     &     &     &     &  \\
			\hline
			\multirow{3}[2]{*}{1} & 1   & \multirow{9}[6]{*}{0.05} & 33  & 52  & 13.29 & 20  & 0.31 & 13.68 & \multirow{9}[6]{*}{0.5} & 30  & 41  & 11.66 & 12  & 0.22 & 11.95 \\
			& 2   &     & 42  & 59  & 15.59 & 17  & 0.26 & 15.99 &     & 57  & 69  & 19.27 & 11  & 0.28 & 19.71 \\
			& 3   &     & 56  & 89  & 17.32 & 34  & 1.02 & 18.56 &     & 33  & 44  & 9.71 & 12  & 0.20 & 9.97 \\
			\cline{1-2}\cline{4-9}\cline{11-16}    \multirow{3}[2]{*}{2} & 1   &     & 29  & 45  & 10.49 & 17  & 0.23 & 10.79 &     & 29  & 34  & 8.70 & 6   & 0.07 & 8.80 \\
			& 2   &     & 44  & 61  & 14.24 & 17  & 0.24 & 14.68 &     & 94  & 99  & 25.98 & 5   & 0.07 & 26.20 \\
			& 3   &     & 32  & 50  & 9.76 & 19  & 0.27 & 10.10 &     & 31  & 42  & 9.18 & 12  & 0.19 & 9.43 \\
			\cline{1-2}\cline{4-9}\cline{11-16}    \multirow{3}[2]{*}{$\infty$} & 1   &     & 21  & 34  & 8.15 & 14  & 0.16 & 8.35 &     & 8   & 9   & 2.34 & 2   & 0.02 & 2.38 \\
			& 2   &     & 37  & 51  & 12.39 & 13  & 0.15 & 12.61 &     & 11  & 15  & 4.23 & 1   & 0.02 & 4.39 \\
			& 3   &     & 21  & 34  & 6.76 & 14  & 0.15 & 6.96 &     & 8   & 9   & 2.05 & 2   & 0.02 & 2.09 \\
			\hline
			\multirow{3}[2]{*}{1} & 1   & \multirow{9}[6]{*}{0.01} & 175 & 262 & 69.13 & 88  & 20.17 & 92.99 & \multirow{9}[6]{*}{0.1} & 143 & 177 & 52.94 & 34  & 2.67 & 56.25 \\
			& 2   &     & 161 & 235 & 62.55 & 73  & 10.56 & 75.28 &     & 232 & 273 & 78.37 & 38  & 5.05 & 84.42 \\
			& 3   &     & 256 & 397 & 76.54 & 142 & 113.29 & 212.22 &     & 412 & 510 & 111.49 & 91  & 48.25 & 165.40 \\
			\cline{1-2}\cline{4-9}\cline{11-16}    \multirow{3}[2]{*}{2} & 1   &     & 128 & 196 & 46.51 & 69  & 8.41 & 56.52 &     &   \multicolumn{6}{c|}{-}\\
			& 2   &     & 145 & 209 & 49.42 & 64  & 6.56 & 57.29 &    &  \multicolumn{6}{c|}{-}\\
			& 3   &     & 139 & 213 & 41.49 & 75  & 10.39 & 53.88 &     & 208 & 265 & 57.93 & 46  & 5.08 & 63.67 \\
			\cline{1-2}\cline{4-9}\cline{11-16}    \multirow{3}[2]{*}{$\infty$} & 1   &     & 93  & 145 & 35.34 & 53  & 3.44 & 39.47 &     & 68  & 82  & 22.59 & 12  & 0.22 & 22.87 \\
			& 2   &     & 107 & 154 & 37.20 & 47  & 2.43 & 40.15 &     &  \multicolumn{6}{c|}{-}\\
			& 3   &     & 87  & 137 & 26.72 & 51  & 3.03 & 30.36 &     &   \multicolumn{6}{c|}{-}\\
			\hline
		\end{tabular}%
	}
\end{table}

\begin{table}[h]
	\centering
	\caption{\rev{Computational results for \Cref{ex:2}}}
	\label{tab:ex2}
	\resizebox{0.6\textwidth}{!}{
		\begin{tabular}{|c|c|ccccccc|}
			\hline
			\multirow{2}[2]{*}{$\epsilon$} & \multirow{2}[2]{*}{$p$} & \multirow{2}[2]{*}{Alg} & \multirow{2}[2]{*}{$\abs{\bar{\mathcal{X}}}$} & \multirow{2}[2]{*}{Opt} & \multirow{2}[2]{*}{T$_\text{opt}$} & \multirow{2}[2]{*}{En} & \multirow{2}[2]{*}{T$_\text{en}$} & \multirow{2}[2]{*}{T} \\
			&     &     &     &     &     &     &     &  \\
			\hline
			\multirow{4}[4]{*}{0.05} & \multirow{2}[2]{*}{1} & 1   & 188 & 310 & 107.26 & 87  & 19.30 & 130.34 \\
			&     & 2   & 157 & 233 & 79.79 & 70  & 9.88 & 91.94 \\
			\cline{2-9}        & \multirow{2}[2]{*}{2} & 1   & 145 & 225 & 70.22 & 76  & 11.46 & 84.07 \\
			&     & 2   & 141 & 206 & 61.72 & 64  & 7.18 & 70.47 \\
			\hline
			\multirow{4}[4]{*}{0.01} & \multirow{2}[2]{*}{1} & 1   &  \multicolumn{6}{c|}{-}\\
			&     & 2   & 772 & 1187 & 401.00 & 340 & 3197.85 & 4276.37 \\
			\cline{2-9}        & \multirow{2}[2]{*}{2} & 1   & 869 & 1421 & 420.01 & 311 & 2302.12 & 3171.41 \\
			&     & 2   & 655 & 957 & 285.04 & 279 & 1529.54 & 2162.02 \\
			\hline
		\end{tabular}%
	}
\end{table}%

\begin{table}[h]
	\centering
	\caption{\rev{Computational results for \Cref{ex:3}}}
	\label{tab:ex3}
	\resizebox{\textwidth}{!}{
		\begin{tabular}{|c|c|c|cccccc|cccccc|}
			\cline{4-15}    \multicolumn{1}{r}{} & \multicolumn{1}{r}{} &     & \multicolumn{6}{c|}{$n$=3}          & \multicolumn{6}{c|}{$n$=9} \\
			\hline
			\multirow{2}[1]{*}{$\epsilon$} & \multirow{2}[1]{*}{$p$} & \multirow{2}[1]{*}{Alg} & \multirow{2}[1]{*}{$\abs{\bar{\mathcal{X}}}$} & \multirow{2}[1]{*}{Opt} & \multirow{2}[1]{*}{T$_\text{opt}$} & \multirow{2}[1]{*}{En} & \multirow{2}[1]{*}{T$_\text{en}$} & \multirow{2}[1]{*}{T} & \multirow{2}[1]{*}{$\abs{\bar{\mathcal{X}}}$} & \multirow{2}[1]{*}{Opt} & \multirow{2}[1]{*}{T$_\text{opt}$} & \multirow{2}[1]{*}{En} & \multirow{2}[1]{*}{T$_\text{en}$} & \multirow{2}[1]{*}{T} \\
			&     &     &     &     &     &     &     &     &     &     &     &     &     &  \\
			\hline
			\multirow{6}[2]{*}{10} & \multirow{3}[1]{*}{2} & 1   & 502 & 943 & 285.43 & 132 & 100.12 & 401.95 & 1561 & 2754 & 1159.60 & 225 & 753.05 & 2046.88 \\
			&     & 2   & 958 & 3924 & 1194.30 & 137 & 122.28 & 1339.09 & 1770 & 4213 & 1819.70 & 218 & 733.58 & 2682.37 \\
			&     & 3   & 305 & 965 & 259.85 & 164 & 211.23 & 502.43 & 2718 & 4520 & 1772.08 & 259 & 1324.24 & 3295.96 \\
			\cline{2-15}        & \multirow{3}[1]{*}{inf} & 1   & 197 & 592 & 179.71 & 106 & 41.61 & 227.61 & 1231 & 2106 & 901.81 & 152 & 150.43 & 1076.14 \\
			&     & 2   & 199 & 1206 & 390.99 & 100 & 36.67 & 432.99 & 3638 & 9222 & 4045.45 & 166 & 219.60 & 4301.63 \\
			&     & 3   & 180 & 586 & 157.46 & 101 & 33.60 & 196.24 & 2628 & 5057 & 1994.39 & 164 & 202.33 & 2231.07 \\
			\hline
			\multirow{6}[3]{*}{5} & \multirow{3}[1]{*}{2} & 1   & 1178 & 3127 & 932.65 & 245 & 1059.10 & 2175.25 & 4461 & 7968 & 3371.91 & 390 & 8008.67 & 12795.23 \\
			&     & 2   & 1207 & 5557 & 1702.79 & 259 & 1488.53 & 3441.28 & 7052 & 15662 & 6546.29 & 409 & 9908.03 & 18268.17 \\
			&     & 3   & 579 & 3932 & 1049.79 & 309 & 3046.66 & 4529.37 & 7046 & 11149 & 4307.21 & 476 & 16898.76 & 23603.25 \\
			\cline{2-15}        & \multirow{3}[2]{*}{inf} & 1   & 412 & 1740 & 526.33 & 185 & 325.59 & 907.19 & 3153 & 4538 & 1889.06 & 294 & 2164.45 & 4390.16 \\
			&     & 2   & 465 & 2655 & 837.24 & 188 & 374.95 & 1268.01 & 3570 & 8155 & 3482.44 & 305 & 2589.99 & 6470.71 \\
			&     & 3   & 342 & 1412 & 380.10 & 185 & 352.70 & 787.48 & 3146 & 4712 & 1845.91 & 300 & 2455.39 & 4663.98 \\
			\hline
		\end{tabular}%
	}
\end{table}%

In line with the theory, Tables~\ref{tab:ex1}-\ref{tab:ex3} illustrate that Opt, En as well as T increase when a smaller approximation error is used, irrespective of the algorithm considered.


\rev{According to \Cref{tab:ex1}, for $p=1$ in terms of all performance measures, Algorithms \ref{alg} and \ref{alg1} perform better than Algorithm 3, except for $q=4, \epsilon =0.5$ for which Algorithms \ref{alg} and 3 have similar performances. For $p\in\{2,\infty\}$, Algorithms \ref{alg} and 3 perform similar to each other and better than \Cref{alg1} in terms of Opt, T$_\text{opt}$ and T. When we compare Algorithms \ref{alg} and \ref{alg1}, we observe that  \Cref{alg1} solves a larger number of optimization problems (Opt) compared to \Cref{alg} in all settings except $p=1, \epsilon = 0.01$. The reason may be that the former algorithm deals with a higher number of vertices, coming from the intersection of $\bar{\mathcal{P}}^\out_k$ with $\bd S$, $k \geq 0$ (line 19 of \Cref{alg1}).}

\rev{\Cref{tab:ex2} indicates that, in solving \Cref{ex:2}, Algorithm \ref{alg1} performs better than \Cref{alg} with respect to all indicators.} 
	
\rev{Finally, for \Cref{ex:3} when we compare the performances under $n=3$ with $p=2$, \Cref{alg} works better compared to the others in terms of Opt, \rev{En} and T. Under $n=3$ with $p=\infty$, the same holds for Algorithm 3, see \Cref{tab:ex3}. However, under $n=9$, \Cref{alg} performs better than the others in all instances.}

\rev{When we compare the results for \Cref{ex:3}(a) and (b), we observe that even with the same precision level $\epsilon$, the number of minimizers is at least twice as and generally much higher than the number of minimizers in Example 8.3(a), which also affects the total times. The reason may be that due to the increase in the dimension of the feasible region and the structure of the objective functions, the range of the objective function changes and the difficulty of the problem increases in Example 8.3(b).}

From the results of the test problems above, we deduce a comparable performance of our proposed algorithms compared to Algorithm 3. 

\rev{Next,} we consider \Cref{ex:1} for $q\in \{2,3\}$ with different ordering cones than the positive orthant, see \Cref{tab:ex1cones} and \Cref{fig:3dcones}. These cones are given below in terms of their generating vectors:
\begin{align}
C_1&=\conv \cone \{(1, 2)^\mathsf{T}, (2, 1)^\mathsf{T}\},\notag \\
C_2&=\conv \cone \{(2, -1)^\mathsf{T}, (-1, 2)^\mathsf{T}\},\notag \\ 
C_3&=\conv \cone \{(4, 2, 2)^\mathsf{T}, (2, 4, 2)^\mathsf{T}, (4, 0, 2)^\mathsf{T}, (1, 0, 2)^\mathsf{T}, (0, 1, 2)^\mathsf{T}, (0, 4, 2)^\mathsf{T}\},\notag  \\
C_4&=\conv \cone \{(-1, -1, 3)^\mathsf{T}, (2, 2, -1)^\mathsf{T}, (1, 0, 0)^\mathsf{T}, (0, -1, 2)^\mathsf{T}, (-1, 0, 2)^\mathsf{T}, (0, 1, 0)^\mathsf{T}\}.\notag 
\end{align}
Note that we have $C_1 \subsetneq \R^2_+ \subsetneq C_2=C_1^+$ and $C_3$\footnote{The same cone is used as a dual cone in \cite[Example~9]{lohne2016equivalence}.} $\subsetneq \R^3_+ \subsetneq C_4=C_3^+$. We solve these examples with Algorithms \ref{alg} \rev{and} \ref{alg1}, where the norm in \eqref{eq:P2(v)} is the $\ell_2$ norm. As before, due to the limitations of \emph{bensolve tools}, \Cref{tab:ex1cones} does not show the result \rev{for \Cref{alg1} when the ordering cone is} $C_3$ \rev{and $\epsilon=0.01$}. According to \Cref{tab:ex1cones}, for \rev{$C_1$ and} $C_2$, Algorithms \ref{alg} \rev{and \ref{alg1}} are comparable in terms of T.
\rev{For $C_3$ with $\epsilon=0.05$, \Cref{alg1} gives smaller T.}
However, for $C_4$, \Cref{alg} \rev{has better runtime.} 

\begin{figure}[h]
	\centering
	\subfloat[$C_3$]{\label{fig:3d_C3}\includegraphics[width=2.5in]{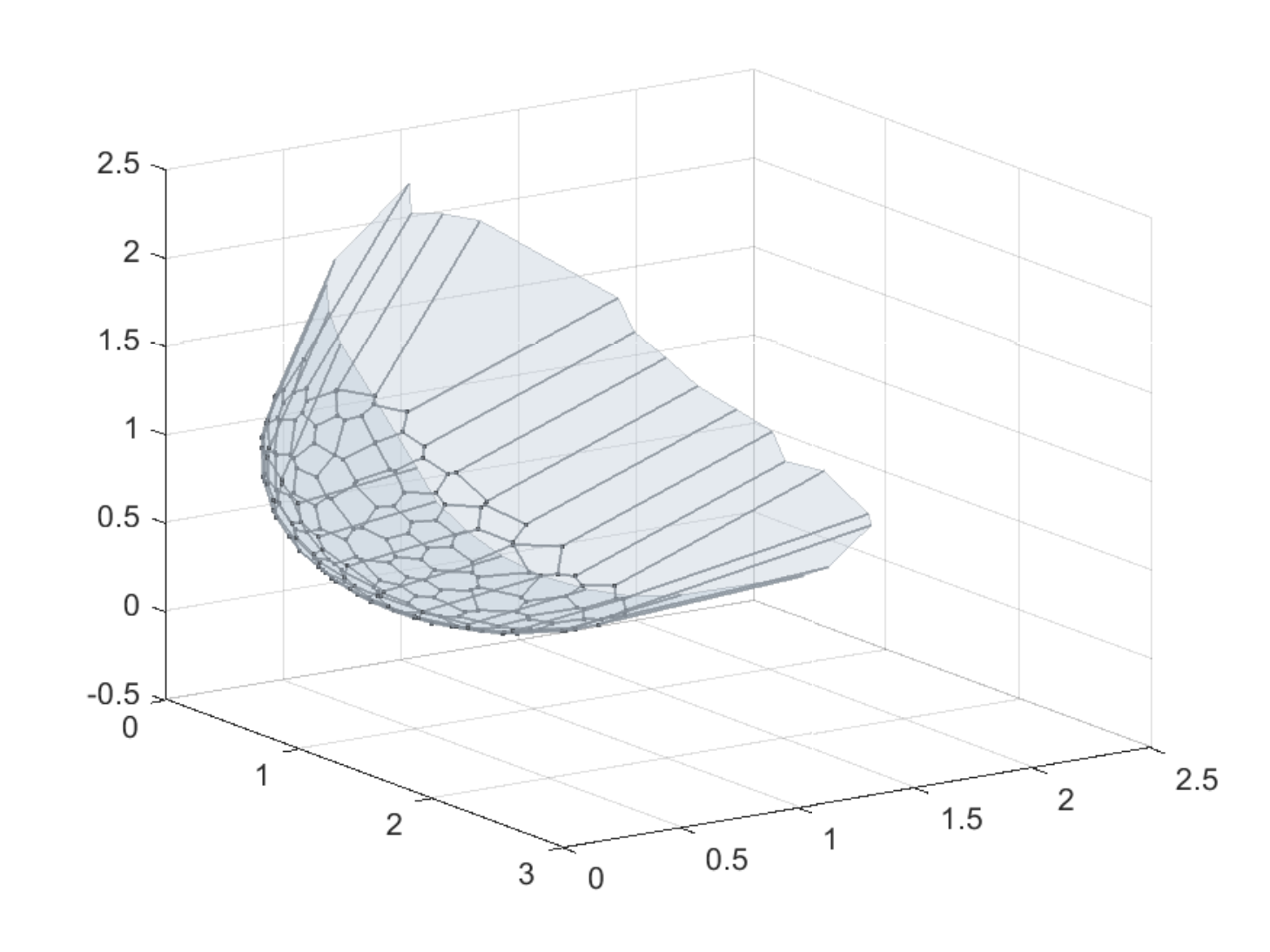}}
	\subfloat[$C_4$]{\label{fig:3d_C4}\includegraphics[width=2.5in]{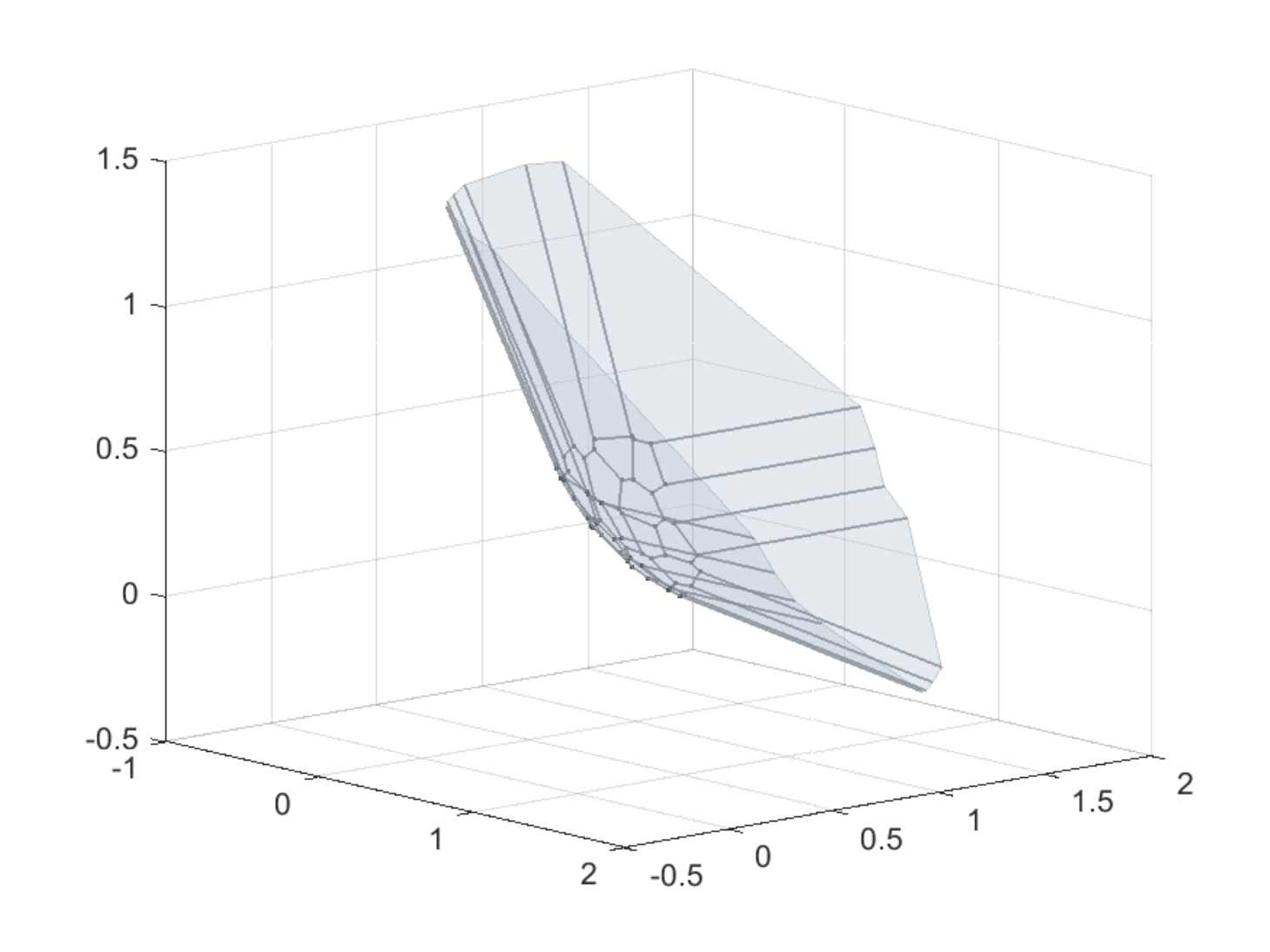}}
	\caption{Outer approximations obtained from \Cref{alg} using $\ell_2$ norm for \Cref{ex:1} with $q=3$ under different cones}
	\label{fig:3dcones}
\end{figure}

\begin{table}[h]
	\centering
	\caption{\rev{Computational results for \Cref{ex:1} with different ordering cones and $p=2$}}
	\label{tab:ex1cones}
	\resizebox{\textwidth}{!}{
		\begin{tabular}{|c|c|c|cccccc|c|c|cccccc|}
			\cline{2-17}    \multicolumn{1}{r|}{} & \multicolumn{1}{r}{} & \multicolumn{7}{c|}{$q$ = 2}              & \multicolumn{1}{r}{} & \multicolumn{7}{c|}{$q$ = 3} \\
			\hline
			\multirow{2}[2]{*}{Alg} & \multirow{2}[2]{*}{$\epsilon$} & \multirow{2}[2]{*}{$C$} & \multirow{2}[2]{*}{$\abs{\bar{\mathcal{X}}}$} & \multirow{2}[2]{*}{Opt} & \multirow{2}[2]{*}{T$_\text{opt}$} & \multirow{2}[2]{*}{En} & \multirow{2}[2]{*}{T$_\text{en}$} & \multirow{2}[2]{*}{T} & \multirow{2}[2]{*}{$\epsilon$} & \multirow{2}[2]{*}{$C$} & \multirow{2}[2]{*}{$\abs{\bar{\mathcal{X}}}$} & \multirow{2}[2]{*}{Opt} & \multirow{2}[2]{*}{T$_\text{opt}$} & \multirow{2}[2]{*}{En} & \multirow{2}[2]{*}{T$_\text{en}$} & \multirow{2}[2]{*}{T} \\
			&     &     &     &     &     &     &     &     &     &     &     &     &     &     &     &  \\
			\hline
			1   & \multirow{4}[4]{*}{0.005} & \multirow{2}[2]{*}{$C_1$} & 19  & 34  & 6.95 & 16  & 0.12 & 7.12 & \multirow{4}[4]{*}{0.05} & \multirow{2}[2]{*}{$C_3$} & 62  & 89  & 20.54 & 28  & 0.79 & 21.50 \\
			2   &     &     & 21  & 36  & 7.51 & 15  & 0.11 & 7.69 &     &     & 57  & 77  & 18.27 & 18  & 0.35 & 18.82 \\
			\cline{1-1}\cline{3-9}\cline{11-17}    1   &     & \multirow{2}[2]{*}{$C_2$} & 6   & 9   & 1.87 & 4   & 0.02 & 1.91 &     & \multirow{2}[2]{*}{$C_4$} & 22  & 29  & 6.57 & 8   & 0.08 & 6.68 \\
			2   &     &     & 8   & 11  & 2.30 & 3   & 0.02 & 2.36 &     &     & 28  & 34  & 7.90 & 4   & 0.04 & 8.01 \\
			\hline
			1   & \multirow{4}[4]{*}{0.001} & \multirow{2}[2]{*}{$C_1$} & 37  & 69  & 14.39 & 32  & 0.45 & 14.96 & \multirow{4}[4]{*}{0.01} & \multirow{2}[2]{*}{$C_3$} & 229 & 346 & 80.66 & 117 & 62.00 & 154.04 \\
			2   &     &     & 36  & 67  & 14.09 & 31  & 0.44 & 14.69 &     &     &     \multicolumn{6}{c|}{-}\\
			\cline{1-1}\cline{3-9}\cline{11-17}    1   &     & \multirow{2}[2]{*}{$C_2$} & 10  & 17  & 3.44 & 8   & 0.05 & 3.52 &     & \multirow{2}[2]{*}{$C_4$} & 71  & 107 & 24.99 & 37  & 1.45 & 26.78 \\
			2   &     &     & 12  & 19  & 4.04 & 7   & 0.04 & 4.13 &     &     & 90  & 123 & 28.94 & 31  & 1.05 & 30.31 \\
			\hline
		\end{tabular}
	}%
\end{table}%

\rev{We conclude this section by a remark that illustrates the necessity of intersecting $\bar{\mathcal{P}}_k^{\out}$ with $S$ in \Cref{alg1}.}

	\begin{remark}\label{rem:outsideS}
	\rev{As noted before in \Cref{sec:alg2}, it is possible that some vertices of $\bar{\mathcal{P}}_k^{\out}$ falls outside $S$. Consider \Cref{ex:1} with $q=3$. Note that $\Gamma(\mathcal{X})$ is the unit ball centered at $e\in\R^3$, and $\mathcal{P}_0$ is the positive orthant. In the initialization phase of \Cref{alg1}, we obtain $S = \{ y\in\R^3 \mid \bar{w}^{\mathsf{T}}y \leq 3.56 \}$, where $\bar{w}=\frac{1}{\sqrt{3}}e$. 
	For illustrative purposes, consider the supporting halfspace $\mathcal{H} = \{ y\in\R^q \mid w^{\mathsf{T}}y \geq 0.68 \}$ of the upper image, where the normal direction is $w = (1,1,0.1)^\mathsf{T}$. This would support the upper image at the $C$-minimal point $(0.2947,0.2947,0.9295)^{\mathsf{T}}$. Note that $\bd \mathcal{H}$ intersects with $\mathcal{P}_0$ to give three vertices: $v^1=(0.68, 0, 0)^\mathsf{T}$, $v^2=(0, 0.68, 0)^\mathsf{T}$, $v^3 = (0, 0, 6.8)^\mathsf{T}$. Clearly, $v^1,v^2\in S$ and $v^3\notin S$. Moreover, as stated in \Cref{rem:Soutside} the approximation generated by $v^1,v^2$, namely, $\conv ( \{v^1,v^2\}) + \R_+^3$ does not contain the upper image; see \Cref{fig:Projection}.}	

\begin{figure}[tbhp]
	\centering
	\includegraphics[width=2.7in]{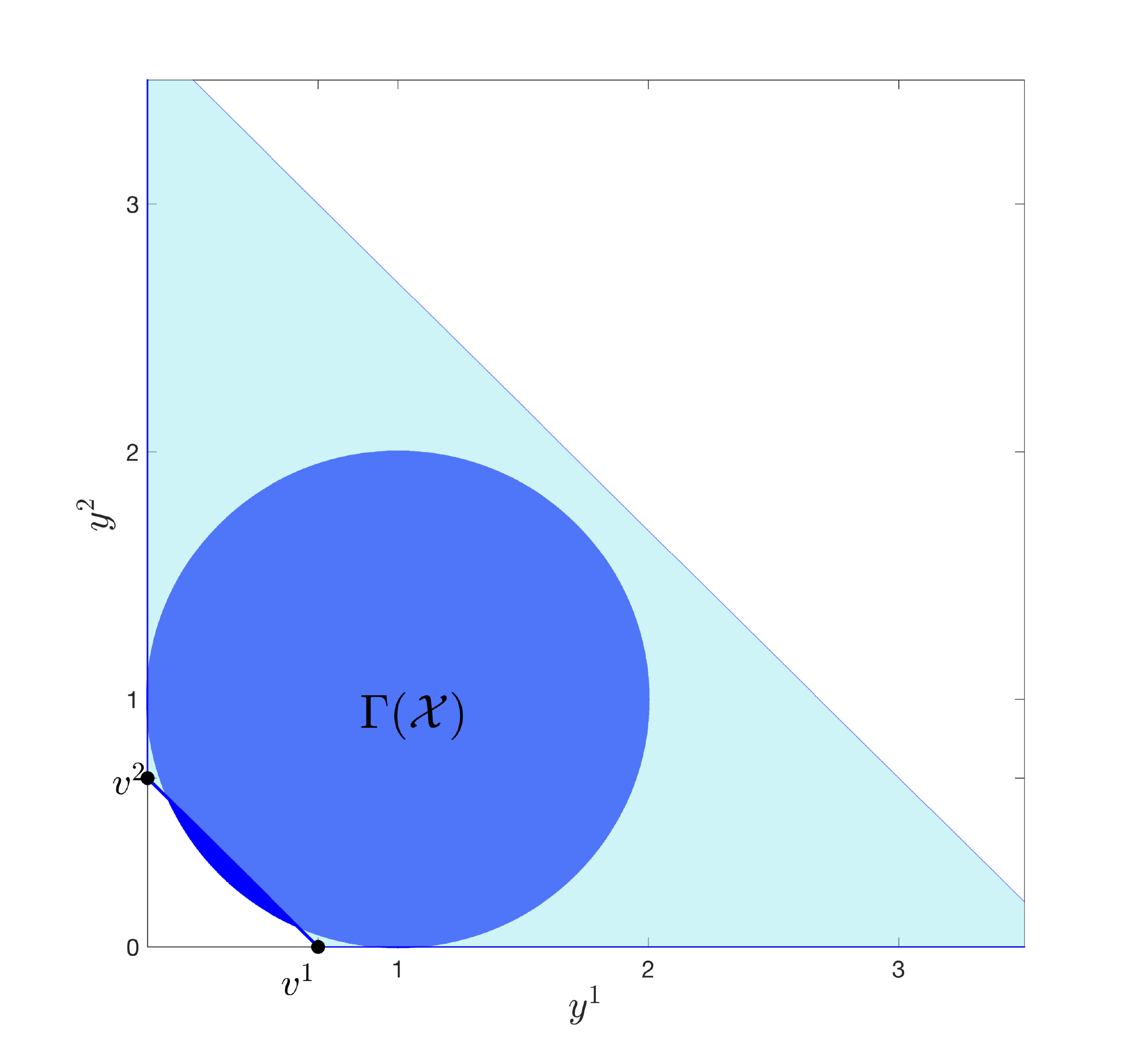}
	\caption{\rev{Projections of $\Gamma(\mathcal{X})$ (dark blue), $\conv ( \{v^1,v^2\}) + \R_+^3$ (light blue) from \Cref{rem:outsideS} on the $y_3=0$ plane}}
	\label{fig:Projection}
\end{figure}

\rev{We practically encounter vertices which fall outside $S$, for instance, while running \Cref{alg} for \Cref{ex:1} with $q=3$, $p=2$ and $\epsilon=0.01$. 
	\Cref{fig:VoutsideS_1} shows the outer approximation, after iteration $k =37$, with one of the vertices outside $\mathcal{P}_0^{\text{out}} \cap S$.}

\begin{figure}[tbhp]
	\centering
	\includegraphics[width=4in]{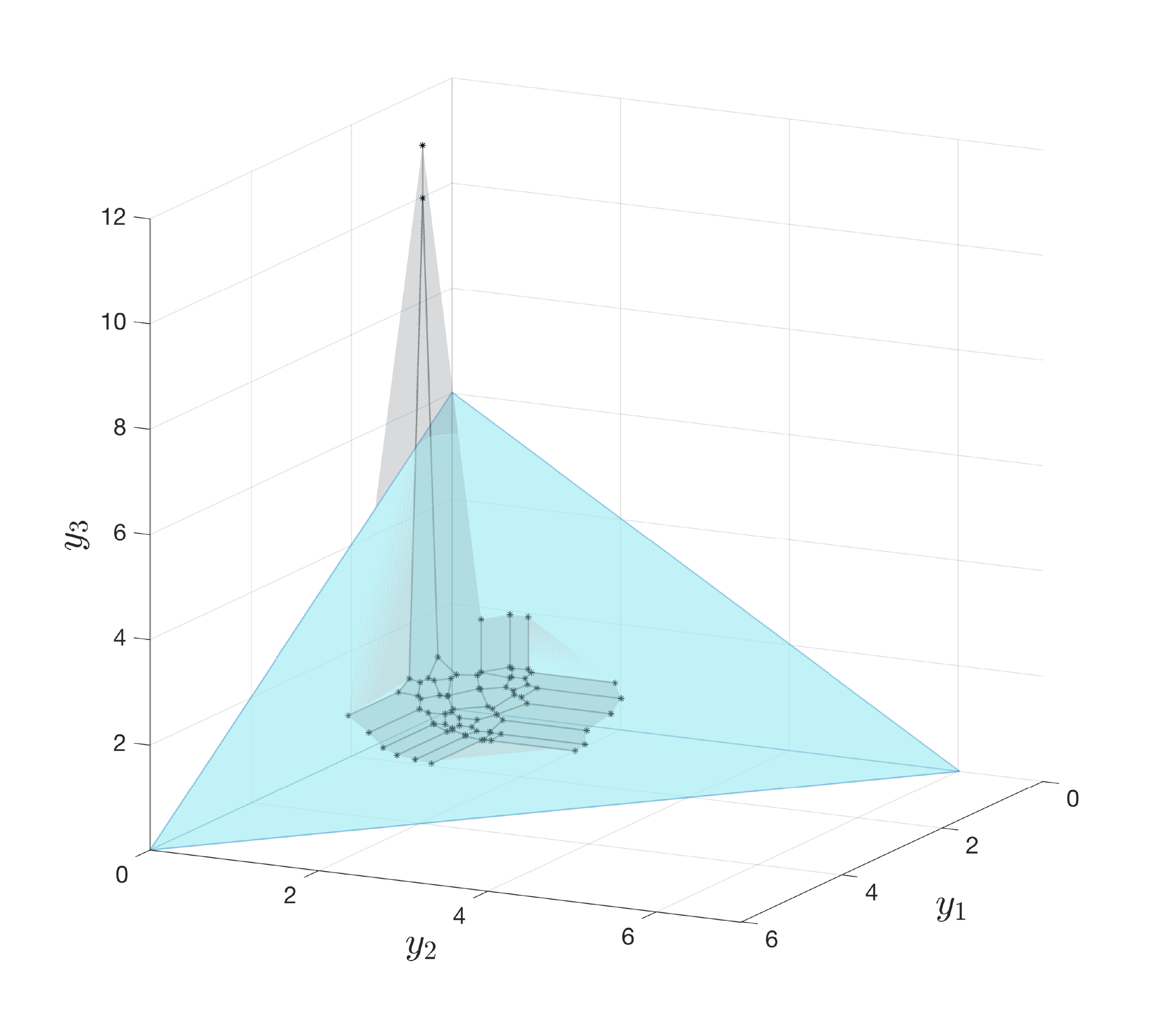}
	\caption{\rev{Outer approximation $\mathcal{P}_k^{\text{out}}$, $k = 37$ (vertices and representative points on unbounded edges indicated by black markers) for \Cref{ex:1} obtained using \Cref{alg}  (for $p=2$ and $\epsilon=0.01$) has one vertex outside $S$ (blue).}}
	\label{fig:VoutsideS_1}
\end{figure}

\end{remark}



\section{Conclusions}
\label{sec:conclusions}

In this study, we have proposed an algorithm for CVOPs which is based on a norm-minimizing scalarization. It is different from the similar class of algorithms available in the literature in the sense that it does not need a direction parameter as an input. We have also proposed a modification of the algorithm and proved its finiteness under the assumption of compact feasible region. \rev{Using benchmark test problems, the computational performance of the new algorithms is found to be comparable to a CVOP algorithm in the recent literature which uses the Pascoletti-Serafini scalarization.}


\singlespacing
\bibliographystyle{plain}
\bibliography{PAACVOP.bib}

\end{document}